\documentclass[11pt]{article}
\usepackage[a4paper,margin=2.8cm]{geometry}

\usepackage[utf8]{inputenc}
\usepackage[T1]{fontenc}

\usepackage{amsmath, amssymb, amsfonts, amsthm}
\usepackage{mathrsfs}
\usepackage{bbm}

\usepackage{graphicx}
\usepackage{booktabs}
\usepackage{multirow}
\usepackage{subcaption}

\usepackage{algorithm}
\usepackage{algpseudocode}
\usepackage{listings}

\usepackage{xcolor}
\usepackage{nicefrac}
\usepackage{microtype}
\usepackage{url}
\usepackage{hyperref}
\usepackage[title]{appendix}

\hypersetup{
	colorlinks=true,
	linkcolor=blue,
	citecolor=blue,
	urlcolor=blue
}

\theoremstyle{plain}
\newtheorem{theorem}{Theorem}

\newtheorem{corollary}[theorem]{Corollary}

\theoremstyle{definition}

\newtheorem{condition}{Condition}

\theoremstyle{remark}
\newtheorem{remark}{Remark}

\newcommand{\alms}{a_{\ell,m;s}}
\newcommand{\almstilde}{\tilde{a}_{\ell,m;s}}
\newcommand{\almE}{a_{\ell,m;E}}
\newcommand{\almB}{a_{\ell,m;B}}

\newcommand{\Ylms}{Y_{\ell,m;s}}

\newcommand{\ClE}{C_{\ell,E}}
\newcommand{\ClB}{C_{\ell,B}}
\newcommand{\Cls}{C_{\ell;s}}

\newcommand{\summ}{\sum_{m=-\ell}^{\ell}}
\newcommand{\sumlm}{\sum_{\ell \geq s}\sum_{m=-\ell}^{\ell}}

\newcommand{\Stwo}{\mathbb{S}^2}
\newcommand{\Ex}{\mathbb{E}}
\newcommand{\diff}{\mathrm{d}}

\newcommand{\Jac}[3]{P_{#1}^{\left(#2,#3\right)}}
\newcommand{\Amat}{\tau_{s}\left( \ell,m;u,v \right)}

\begin{document}

\title{Aliasing Effects for Samples of Spin Random Fields on the Sphere}

\author{
	Claudio Durastanti \\
	\small Department of Basic and Applied Sciences for Engineering \\
	\small Sapienza University of Rome \\
	\small \texttt{claudio.durastanti@uniroma1.it}
}

\date{}

\maketitle

\begin{abstract}
	This paper investigates aliasing effects emerging from the reconstruction from
	discrete samples of spin spherical random fields defined on the two-dimensional
	sphere. We determine the location in the frequency domain and the intensity of the
	aliases of the harmonic coefficients in the Fourier decomposition of the spin random
	field and evaluate the consequences of aliasing errors in the angular power spectrum
	when the samples of the random field are obtained by using some very popular sampling
	procedures on the sphere, namely the equiangular and the Gauss--Jacobi sampling
	schemes. Finally, we demonstrate that band-limited spin random fields are free from
	aliases, provided that a sufficiently large number of nodes is used in the selected
	quadrature rule.

\bigskip


\medskip


\end{abstract}

\section{Introduction} 
\subsection{Motivations}
A spin spherical random field refers to a type of random field defined on the surface of a sphere that possesses certain rotational properties under transformations. In simpler terms, it is a mathematical construct used to model random variations or fluctuations over the surface of a sphere, where the fluctuations exhibit specific rotational symmetries (see, for example, \cite[Chapter 12]{marpec11} and the references therein). The term spin refers to the transformation properties of the field under rotations of the sphere. These fields find applications in various areas, including astrophysics, cosmology, and geophysics, where phenomena are often observed and studied on spherical surfaces. Spin random fields on the sphere have been originally introduced in \cite{np66}, to be further expanded upon by \cite{gm11,gm09} (additional insights in \cite[Chapter 12]{marpec11}), forging connections between the concept of spin-s quantities and sections of the spin-s line bundle on the sphere. Subsequently, numerous works including \cite{br14,ls12,lmrs22,mal11, mal13, ps24,stec22} delved into these geometric objects, propelled by theoretical curiosity and their relevance in statistical applications and cosmological studies \cite[Section 1.2]{marpec11}.\\
In fact, spin spherical functions have gained considerable attention in cosmology over the past three decades, in particular within the realm of Cosmic Microwave Background (CMB) polarization data analysis. This concept can be characterized as observations of random ellipses lying on the tangent planes at various points across the celestial sphere. Mathematically, this can be articulated as sections of vector bundles on the sphere and are commonly interpreted probabilistically as single realizations of random sections from Gaussian spin bundles, generalizing the examination of CMB temperature data, which are construed as realizations of Gaussian scalar-valued random fields on the sphere, see for example \cite{kks97}. \\  This poses a significant challenge in the analysis of astrophysical data, since the investigation of CMB serves as a fundamental instrument in probing Big Bang models and ascertaining key cosmological parameters. Consequently, it has captivated huge interest over the last two decades, leading to the launch of two major satellite missions, NASA’s WMAP and ESA’s Planck (for more details, the reader is referred to \cite{planck20} and the references therein).  Further investigations are currently ongoing by the satellite mission Euclid, led by the European Space Agency and designed also to investigate the nature of dark energy and dark matter. In fact, the same mathematical framework plays a pivotal role in other crucial cosmological observations, notably in weak gravitational lensing data, affecting the apparent shapes of distant galaxies due to concentrations of matter. Euclid aims to achieve this through precise measurements of the shapes and redshifts of galaxies across the sky, employing a combination of imaging and spectroscopic techniques, then described as spherical spin functions (see \cite{euclid20,euclid24}). Several statistical methodologies have already been devised in the existing literature also to address some of those cosmological applications (see for example \cite{dgm12,glm2009}). \\
A common method for mathematically representing spin random fields on the sphere involves utilizing spin-weighted representations of special group of rotations $SO(3)$, the set of orthonormal $3 \times 3$ matrices with unit determinant, which allows a spectral representation of the field. Let $x=\left(\vartheta, \varphi\right)$, $\vartheta\in\left[0,\pi\right]$, $\varphi\in\left[\left.0,2\pi\right)\right.$ be a generic location on the sphere and $\diff x = \sin \vartheta \diff \vartheta \diff \varphi$ the spherical uniform measure. The angle $\vartheta$ is usually referred to as colatitude, while $\varphi$ is labeled as longitude. A spin $s$-weighted random field $T_s$, $s \in \mathbb{N}$, can be defined via the following function on the sphere: 
\begin{equation}\label{eq:introfield}
	T_s\left(\vartheta,\varphi\right) = \sumlm \alms \Ylms\left(\vartheta,\varphi\right),
\end{equation}
where $\left\{\Ylms: \ell \geq s, m=-\ell,\ldots, \ell\right\}$ are the so-called spin weighted spherical harmonics and $\left\{\alms: \ell \geq s, m=-\ell,\ldots, \ell\right\}$ are corresponding random harmonic coefficients, given by 
\begin{equation}\label{eq:introcoeff}
	\alms=\int_{\mathbb{S}^2} T_s \left(x\right) \bar{Y}_{\ell,m;s}\left(x\right) \diff x, 
\end{equation} 
containing all the statistical information of $T_s$ (see for example \cite{marpec11, vmk} and Section \ref{sec:back} below). \\
However, calculating the integral \eqref{eq:introcoeff} explicitly is often impractical in many experimental scenarios. Typically, measurements of $T_s$ can only be gathered from a finite set of locations on the spherical surface 
\begin{equation}\notag 
	\left\{x_k=\left(\vartheta_k,\varphi_k\right):k=1,\ldots,K\right\},
\end{equation}
where $K$ is the sample size. Consequently, for any choice of $\ell$ and $m$, 
the integral yielding the harmonic coefficient $\alms$ is approximated by a weighted sum of a finite number of samples of the random field $\left\{T_s\left(x_k\right):k=1,\ldots,K\right\}$. As commonly understood in the literature, an exact reconstruction of the harmonic coefficients through finite sums is feasible when dealing with band-limited random processes. These processes are characterized by a bandwidth $L_0$ such that all harmonic coefficients for $\ell>L_0$ are zero and, then, the sum in \eqref{eq:introfield} is finite. Thus, by appropriately selecting a sampling theorem and the number of sampling points, exact reconstruction for the non-zero coefficients is achievable (see, for instance, \cite{m07}). Further discussion on this matter will be provided in Section \ref{sec:alias}.
Nevertheless, if the random field is not band-limited or if the sampling theorem is improperly chosen, approximating the integral in \eqref{eq:introcoeff} with a finite sum can result in aliasing errors. Aliasing errors occur when different coefficients become indistinguishable aliases of one another (see, for example, \cite{m07,stobur93}). The set of coefficients acting as aliases depends on the chosen sampling procedure chosen. The concept of aliasing, also known as confounding, originates from signal processing theory and related disciplines. In essence, aliasing makes different signals indistinguishable when sampled. It occurs when the reconstruction of the signal from its samples deviates from the original continuous signal (see, for example, \cite[Chapter 1]{pm96}).\\
The phenomenon of aliasing in the harmonic expansion of a 2-dimensional spherical random field has been studied by \cite{ln97}. Approximating spherical harmonic coefficients via discrete sums is of theoretical interest (see for example \cite{dmp12,dmt24}), but also relevant by the practical standpoint. Aliasing effects have indeed been examined to study stochastic temperature fields on a sphere by \cite{lns97}. This approach has then been extended to the $d$-dimensional case in \cite{dp19}. In the two-dimensional case, band-limited random fields over $\mathbb{S}^2$, which can be roughly seen as linear combinations of a finite number of standard spherical harmonics, can be uniquely reconstructed with a sufficiently large sample size. Furthermore, an explicit definition of the aliasing function, a fundamental tool for identifying the aliases of a given harmonic coefficient, is established when the sampling is based on the combination of a Gauss-Legendre quadrature formula and a trapezoidal rule, initially introduced in \cite{sku86} and further elaborated by \cite{dh2003}. For an exhaustive and more detailed review, the reader is referred to \cite{MW11,wjv07}. In the $d$-dimensional scenario, these findings have been expanded thanks to the hyperspherical harmonics representation to apply to the $\mathbb{S}^d$ case, and the aliasing function has been formulated specifically for sampling methods involving a combination of $d-1$ Gauss-Gegenbauer quadrature formulas along with a trapezoidal rule. 
In numerous practical applications, this sampling approach stands out as the most convenient method for conducting numerical analysis on the sphere (see, for instance, \cite{at12, stobur93, sze75}). Moreover, it has been demonstrated in \cite{dette05, dette19} that experimental designs utilizing this sampling methodology on $\mathbb{S}^2$ and $\mathbb{S}^d$ surfaces prove to be the most efficient among all approximate designs for regression tasks involving spherical predictors. 
This approach is currently used in the Euclid mission, to investigate the nature of dark energy and dark matter by accurately measuring the accelerated expansion of the Universe. Gauss-Legendre quadrature can be utilized in various aspects of data analysis and numerical simulations involved in the mission, including numerical integration in cosmological models (e.g., calculating integrals over redshift space to model the distribution of dark matter and dark energy), and analysis of observational data (e.g, integrating over the observed light curves or spectra from distant galaxies and quasars to infer their properties), see for example \cite{euclid20, euclid24}.
\subsection{Aim of the paper}
In this paper, our objective is to identify the locations in the frequency domain and the intensities of the aliases of the harmonic coefficients in the Fourier decomposition of the spin random field, and to evaluate the impact of aliasing errors on the angular power spectrum when samples of the random field are obtained using some of the most popular sampling procedures on the sphere: the equiangular and Gauss-Jacobi sampling schemes. By analyzing these sampling methods, we illustrate how the presence of aliases can distort the harmonic coefficients angular power spectrum, leading to inaccuracies in the representation of the spin random field. Furthermore, we demonstrate that for band-limited spin random fields, aliases are completely eliminated, provided that a sufficiently large number of nodes are used in the chosen quadrature rule. This finding underscores the importance of selecting appropriate sampling parameters to ensure accurate and alias-free reconstructions of the spin random fields, thereby preserving the accuracy of the angular power spectrum. We define an aliasing function under specific conditions related to the sampling scheme, with the goal of emulating ideal experimental settings and replicating the environment in which aliasing effects might arise. This function aims to determine the locations of aliases and quantify their contributions to the reconstruction of the theoretical harmonic coefficients. More in detail, our analysis operates under the following premise: a spin spherical random field $T_s$ is observed at a finite set of locations $\left\{x_k:k=1,\ldots,K \right\}$ named sampling points, each associated with weights $\left\{w_k: k=1,\ldots,K\right\}$. For any given set of harmonic numbers $\ell$ and $m$, the approximated—or aliased—harmonic coefficient is expressed as: 
\begin{equation*}\notag
	\tilde{a}_{\ell,m;s} = \sum_{\ell^\prime,m^\prime} \tau_s\left(\ell,m;\ell^\prime, m^\prime\right) a_{\ell^\prime, m^\prime;s},
\end{equation*}
where $ \tau_s\left(\ell,m;\ell^\prime, m^\prime\right) $ denotes the so-called aliasing function (also known as matrix) depending on the sample $\left\{x_k,w_k:k=1,\ldots,K \right\}$. A coefficient $ a_{\ell^\prime, m^\prime;s}$ is considered an alias of $\alms$ with intensity $ \left \vert \tau_s\left(\ell,m;\ell^\prime, m^\prime\right)\right \vert$ if $\tau_s\left(\ell,m;\ell^\prime, m^\prime\right) \neq 0$. \\
Initially, we examine the general structure of the aliasing function under the simple assumption that the sampling scheme is separable concerning the angular coordinates. A sampling scheme is considered separable if distinct discretization procedures are employed for each angular coordinate. We delve into the explicit form of this function and proceed to identify aliases, assuming that the sampling has been done following some among the most common spherical experimental designs. On the one hand, we use the spin spherical uniform design as the sampling method, which involves a combination of the Gauss-Jacobi quadrature formula and the trapezoidal rule. On the other, we use the equiangular approach presented in \cite{dh2003}, where both longitude and colatitude are sampled using a trapezoidal scheme. This approach has been optimized by associating spherical objects to the 2-dimensional torus via a periodic extension, see among others \cite{HW10, MW11}. For all these methods, we will utilize shared parity properties related to sampling the colatitude to demonstrate that they produce aliased harmonic coefficients, with varying intensities depending on the specific scheme. \\
Subsequently, we analyze the consequences of aliasing effects on the angular power spectrum of a spin random field, which defines the spectral decomposition of its covariance function. Finally, we explore the aliasing effects for band-limited spin random fields, establishing conditions on the sample size to mitigate the aliasing phenomenon.
\subsection{Plan of the paper}
The plan of the paper is as follows.\\ Section \ref{sec:back} collects background information and key results on spin random fields, including their expansion using spin spherical harmonics and their relationship with Jacobi polynomials. This section also provides some hints on the construction of the angular power spectrum for spin random fields. \\Section \ref{sec:unif}  outlines the sampling methods that will be examined for aliasing effects: the Gauss-Jacobi and the equiangular methods. \\ Section \ref{sec:alias} presents our main results on identifying aliases in the harmonic coefficients and angular power spectrum of spin random fields. We show that selecting appropriate sampling parameters can significantly reduce aliases and demonstrate that bandlimited random fields can be then made free from aliases. These results are achieved through the construction of a spin aliasing function, which is designed to identify and measure the significance of aliases. Our analysis assumes the separability of the sampling scheme for the angular variables, building on the methods introduced in the previous section. \\Section \ref{sec:example} provides the reader with a practical example to help the understanding of the main results. \\Section \ref{sec:proof} collects the detailed proofs of the theorems and other results discussed in the previous sections.

\section{Background}\label{sec:back}
This section provides a detailed introduction to spin spherical random fields, including their harmonic decomposition and the analysis of their covariance function via the angular power spectrum using spin spherical harmonics. Additionally, we review the relationship between spin spherical harmonics, rotation matrices on SO(3), and Jacobi polynomials. For further details, the reader is referred to \cite{at12,marpec11} and the references therein.
\subsection{Spin elements on the sphere}
This section aims to provide a review of foundational concepts regarding spin fiber bundles. Our exposition closely follows  \cite{glm2009,gm11} where readers can find further discussion and elaboration  (see also \cite[Chapter 12]{marpec11}).
In accordance with the definition of spin functions introduced by Newman and Penrose in \cite{np66}, a function $g$ has an integer-valued spin weight $s$ (or simply, $g$ is a spin $s$ quantity) if, upon a coordinate transformation where a tangent vector at a point $x \in \Stwo$ is rotated by an angle $\psi$, $g$ transforms as $g^\prime=e^{is\psi}g$. \\
In \cite{gm11}, the mathematical framework is outlined as follows. Consider the North and South poles denoted by 
$\mathcal{N}$ and $\mathcal{S}$ respectively. Let $U_I := \mathbb{S}^2/\left\{\mathcal{N},\mathcal{S}\right\}$, 
be the chart that covers the sphere excluding these two points, utilizing the usual angular coordinates $\left(\vartheta, \varphi\right)$, where $\vartheta \in \left[0,\pi\right]$ and $\varphi \in \left[\left.0,2\pi\right)\right.$. Define the rotated chart $U_{\mathcal{R}} = \mathcal{R}U_I$ , where  $\mathcal{R} \in SO(3)$, the special group of rotations, and label the corresponding coordinates as $\left(\vartheta_\mathcal{R} , \varphi_\mathcal{R} \right)$. For any $x \in \mathbb{S}^2$, 
we can establish a reference direction in the tangent plane at $x$, typically denoted $T_x \left(\mathbb{S}^2\right)$, by considering $\rho_I=\partial/\partial \varphi$, the unitary tangent vector aligned with the circle in the direction where $\vartheta$ is constant and $\varphi$ is increasing.\\
For every $x$ lying in the intersection of the charts corresponding to $U_{\mathcal{R}}$ and $U_I$, we can uniquely quantify the angle associated with a change of coordinate by determining the angle between the reference vector in the map $U_I$, and the one in the rotated chart, denoted as $\rho_{\mathcal{R}} (x) = \partial/\partial \varphi_{\mathcal{R}}$. \\
More generally, given $x \in \Stwo$and two charts $U_{\mathcal{R}_1}$ and $U_{\mathcal{R}_2}$ such that $x \in U_{\mathcal{R}_1} \cap U_{\mathcal{R}_2}$, the angle $\psi_{x,\mathcal{R}_1,\mathcal{R}_2}$ between $U_{\mathcal{R}_1}$ and $U_{\mathcal{R}_2}$,  is defined as the angle between $\rho_{\mathcal{R}_1}(x)$ and $\rho_{\mathcal{R}_2}(x)$, which is independent of any choice of coordinates (see again \cite{gm11}). Now the collection of functions $\left\{f_\mathcal{R}\right\}_{\mathcal{R}\in SO(3)}$ is a spin $s$ function $f_s$ if and only if for all $\mathcal{R}_1,\mathcal{R}_2 \in SO(3)$ and all $x \in U_{\mathcal{R}_1} \cap U_{\mathcal{R}_2}$ it holds that
\begin{equation*}
	f_{\mathcal{R}_2}=e^{is\psi_{x,\mathcal{R}_1,\mathcal{R}_2}}f_{\mathcal{R}_1}.
\end{equation*}
Observe that for $s=0$, we obtain the usual scalar functions. We say that $f_s$ is smooth, that is, $f_s \in C_{s}^{\infty}\left(\mathbb{S}^2\right)$, if for every $\mathcal{R} \in SO (3)$ the application $x \mapsto f_s \left(x\right)$ is smooth, while the functional spaces $L^p_s\left(\mathbb{S}^2\right)$ are then defined as are the sets of spin $s$ functions $f_s$ such that
\begin{equation*}
	\left\| f_s\right\|_{L^p_s\left(\mathbb{S}^2\right)}\left(\int_{\mathbb{S}^2} \left\vert f_s\left(x\right)\right \vert^p \diff x \right)^{\frac{1}{p}}<\infty.
\end{equation*}
Finally, observe that while $f_s (x)$ is a section of the fiber bundle on the sphere, its modulus is a real valued function
independent of the choice of the coordinate system such that the spaces $L^p_s\left(\mathbb{S}^2\right)$ are well defined.
\subsection{Spin spherical harmonics and the Jacobi polynomials}
Consider initially the case $s=0$ to focus on scalar elements on the sphere. It is well-established in the literature (see, for example, \cite{marpec11, at12}) that a complete orthonormal basis for the space of square-integrable functions on the sphere $L^2\left(\mathbb{S}^2\right)$ is provided by the family of spherical harmonics denoted by $\left\{Y_{\ell,m}:\ell \geq 0, m=-\ell,\ldots,\ell\right\}$. Here, $\ell$ is named principal - or multipole - number; it is sometimes also called  frequency. The number $m$ is the azimuth number. The spherical harmonics satisfy the orthogonality property
\begin{equation*}
	\int_{\Stwo} Y_{\ell,m}(x)\bar{Y}_{\ell,m}(x)\diff x = \delta_\ell^{\ell^\prime}\delta_m^{m^\prime}
\end{equation*}
Furthermore, the spherical harmonics serve as eigenfunctions of the spherical Laplacian operator $\Delta_{\mathbb{S}^2}$, 
such that, for any $x \in \Stwo$, 
\begin{equation*}
	\Delta_{\mathbb{S}^2}Y_{\ell,m}(x)=-\ell\left(\ell+1\right)Y_{\ell,m}(x).
\end{equation*} 
Expressed in angular coordinates, they can be represented as the normalized product of a complex exponential function and a so-called Legendre associated function $P_{\ell,m}:[-1,1]\mapsto \mathbb{R}$,
\begin{equation*}
	Y_{\ell,m}(\vartheta,\varphi) = \sqrt{ \frac{2\ell+1}{4\pi} \frac{\left( \ell-m\right)!}{\left( \ell+m\right)!}} P_{\ell,m}\left(\cos \vartheta\right)e^{im\varphi}.
\end{equation*}  
The Legendre associated functions are derived from the Legendre polynomials $P_{\ell}:[-1,1]\mapsto \mathbb{R}$ as follows 
\begin{equation*}
	P_{\ell,m}(t)=\left(1-t^2\right)^{\frac{m}{2}}\frac{d^m}{dt^m}P_{\ell}(t),
\end{equation*}
where $\left\{ P_\ell:\ell\geq 0  \right\}$ class of orthogonal polynomials on $[-1,1]$.
For any $f \in L^2\left(\mathbb{S}^2\right)$, the following harmonic expansion holds
\begin{equation*}
	f(x)=\sum_{\ell \geq 0} \sum_{m=-\ell}^{\ell} a_{\ell,m}Y_{\ell,m}(x), \quad a_{\ell,m}=\int_{\mathbb{S}^2}\overline{Y}_{\ell,m}(x)f(x)\diff x,
\end{equation*}
where $\left\{ a_{\ell,m}: \ell \geq 0, m=-\ell,\ldots,\ell\right\}$ are the harmonic coefficients. \\
In simpler terms, standard scalar functions on the sphere can be viewed as objects lying in the space generated by the column $s=0$ of the Wigner's $D=\left\{D^{\ell}_{m,s} \right\}_{\ell\geq 0;m,s=-\ell,\ldots,\ell}$ matrices. These matrices provide an irreducible representation for the rotation group $SO(3)$, which has been extensively discussed in works such as \cite{marpec11, vmk}. Specifically, it is well-known that the elements $\left\{D^{\ell}_{m,0} \right\}_{\ell\geq 0;m=-\ell,\ldots,\ell}$ in the central column of Wigner's $D$ matrices are directly proportional to the standard spherical harmonics $Y_{\ell,m}$. According to the Peter-Weyl Theorem, each column $s=-\ell,\ldots,\ell$ forms a space of irreducible representations, and these spaces are mutually orthogonal. Naturally, one might inquire about the significance of these additional spaces, to extend even further the analogy between the scalar and spin cases. 
These spaces are indeed closely connected to spin functions. In fact, the spin construction can alternatively be formulated using the spin-weighted representation of the special group of rotations $SO(3)$. The fiber bundle of spin $s$ functions $f_s$ can be described as the spectral expansion in terms of $\left\{D^{\ell}_{m,s} \right\}_{\ell\geq 0;m=-\ell,\ldots,\ell}$, as detailed for example in \cite{glm2009,ls12,mal11,mal13}. \\
For all integers $s$ and $\ell \geq 0$, there exists a differential operator $\eth \bar{\eth}$ such that $-\eth \bar{\eth} D^{\ell}_{m,s} = e_{\ell,s}  D^{\ell}_{m,s} $ where, for $\ell \geq s$, $e_{\ell,s}= (\ell-s)(\ell + s + 1)$ represents the associated sequence of eigenvalues. It's noteworthy that for $s=0$, this formulation reverts to the conventional expressions for the scalar case, as anticipated. (see \cite{gm11}). Let us preliminarily define the spin raising and spin lowering operators; for any spin $s$ valued function $f_s$   
\begin{align*}
	&\eth f_s(\vartheta, \varphi) = -\left(\sin \vartheta\right)^s \left[\frac{\partial}{\partial \vartheta}+i \frac{1}{\sin \vartheta}\frac{\partial}{\partial \varphi} \right] \left(\sin \vartheta\right)^{-s} f_s(\vartheta, \varphi) \\
	&\bar{\eth} f_s(\vartheta, \varphi) = -\left(\sin \vartheta\right)^{-s} \left[\frac{\partial}{\partial \vartheta}-i \frac{1}{\sin \vartheta}\frac{\partial}{\partial \varphi} \right] \left(\sin \vartheta\right)^{s} f_s(\vartheta, \varphi) 
\end{align*}
The operator $\eth$ transforms spin $s$ functions into spin $s + 1$ functions, while $\bar{\eth}$ transforms spin $s$ functions into spin $s-1$ functions
justifies their names. 
Spin spherical harmonics $\left\{\Ylms:\ell\geq s, m=-\ell,\ldots,\ell\right\}$, proportional to the elements $\left\{ D_{m,s} ^{\ell} :\ell\geq s, m=-\ell,\ldots,\ell \right\}$ can be then introduced as the eigenfunctions of  the second-order differential operator $\eth \bar{\eth}$ that extends the spherical Laplacian to spin spaces and they are related to Wigner's D matrices by the formula 
\begin{eqnarray}\notag
	\Ylms\left(\vartheta,\varphi\right)& = &\sqrt{\frac{2\ell+1}{4\pi}}\overline D_{m,-s}^{\ell} \left\{\varphi,\vartheta,0\right\}\\
	\notag
	& = &\sqrt{\frac{2\ell+1}{4\pi}}\left(-1\right)^{s} e^{im\varphi} d_{m,-s}^{\ell}\left(\vartheta\right),
\end{eqnarray}
where $d_{m,-s}^{\ell}$ is the so-called Wigner rotation (or small) matrix. Wigner $d$ and $D$ matrices are related by the following formula
\begin{equation*}
	D^{\ell}_{m,s}(\varphi,\vartheta,\delta)=e^{-im\varphi}  d_{m,s}^{\ell}\left(\vartheta\right)e^{-is\delta},
\end{equation*}
with $\delta\in\left[\left.0,2\pi\right)\right.$ denoting the third Euler angle.
Observe that the spin spherical harmonics are orthonormal, 
\begin{equation*}
	\int_{\mathbb{S}^2} Y_{\ell,m;s} (x) \overline{Y}_{\ell^\prime,m^\prime;s} (x) \diff x =\delta_{\ell}^{\ell^\prime}\delta_{m}^{m^\prime}, 
\end{equation*}
while, by means of the properties of the Wigner's matrices (see \cite{glm2009}), the following summation formula holds for any $x, x^\prime \in \Stwo$ identified as rotations belonging to $SO(3)$,
\begin{equation}\label{eq:summation}
	\sum_{m=-\ell}^{\ell} Y_{\ell,m;s} (x) \overline{Y}_{\ell^\prime,m^\prime;s} (x^\prime)=\frac{2\ell+1}{4\pi}D_{s,s}^{\ell}\left( \psi(x,x^\prime)\right),
\end{equation}
where $\psi(x,x^\prime)$ is the composition of two rotations (see \cite{vmk}), related to the Euler angles for $x$ and $x^\prime$, and where 
\begin{align}\notag
	\overline{Y}_{\ell,m;s}&=(-1)^{s+m}Y_{\ell,-m;-s}\\
	\notag &=(-1)^{m} \sqrt{\frac{2\ell+1}{4\pi}} e^{-im\varphi} d_{-m,s}^{\ell}\left(\vartheta\right)
\end{align}
If $x=x^\prime$, Equation \eqref{eq:summation} becomes
\begin{equation*}
	\sum_{m=-\ell}^{\ell} Y_{\ell,m;s} (x) \overline{Y}_{\ell^\prime,m^\prime;s} (x) =\frac{2\ell+1}{4\pi},
\end{equation*}
see for instance \cite{gn67}.
Moreover, a connection between spin spherical harmonics and the Jacobi polynomials $P_{\nu}^{\alpha,\beta}:\left[-1,1\right]\mapsto \mathbb{R}$ can be established  through the explicit formula of the Wigner rotation matrices, as outlined in \cite{marpec11,vmk},
To simplify the notation, from now on, for any set of integers $z_1,z_2,z_3 \in \mathbb{Z}$ we define 
\begin{equation*}
	h_{z_1}^{z_2}(z_3)=  \left[\frac{(z_3-z_2) ! \left( z_3+z_2\right) !} {\left(z_3 + z_1 \right)!\left(z_3 - z_1 \right)!}\right]^{\frac{1}{2}};
\end{equation*}
the Wigner rotation matrix $d_{m,-s}^{\ell}$ is given by  
\begin{equation}\label{eqn:wigner}
	d_{m,-s}^{\ell}\left(\vartheta\right)=h_{s}^{m}(\ell)\left( \sin \frac{\vartheta}{2}\right)^{m+s}\left( \cos \frac{\vartheta}{2}\right)^{m-s} \Jac{\ell-m}{m+s}{m-s}\left(\cos\vartheta\right),
\end{equation}
The Jacobi polynomials $\Jac{\nu}{\alpha}{\beta}:\left[-1,1\right]\mapsto \mathbb{R}$, also known as hypergeometric polynomials, are a class of orthogonal polynomials with respect to the weight function $\left(1-t\right)^{\alpha}\left(1+t\right)^{\beta}$, $\alpha,\beta>-1$, that is
\begin{align}\label{eqn:orthonorm}
	\int_{-1}^{1}&\left(1-t\right)^{\alpha}\left(1+t\right)^{\beta} \Jac{\nu}{\alpha}{\beta}\left(t\right) \Jac{\nu^\prime}{\alpha}{\beta}\left(t\right) = \delta_{\nu}^{\nu^{\prime}} \Lambda^{(\alpha, \beta)} _\nu,
\end{align}
with the normalization constant $\Lambda^{(\alpha, \beta)} _\nu$ given by
\begin{align*}
	\Lambda^{(\alpha, \beta)} _\nu&=\frac{2^{\alpha+\beta+1}}{2\nu+\alpha+\beta+1}\frac{\Gamma\left(\nu+\alpha+1\right)\Gamma\left(\nu+\beta+1\right)}{\nu!\Gamma\left(\nu+\alpha+\beta+1\right)},
\end{align*}
where $\Gamma(\cdot)$ is the Euler Gamma function.\\
Observe that Jacobi polynomials satisfy the following parity relation
\begin{equation}\label{eq:parity}
	P_{\nu}^{\alpha,\beta}(-t)= (-1)^{\nu}P_{\nu}^{\beta,\alpha}(t), \quad \text{for }t \in \left[-1,1\right].
\end{equation} 
This relation is mirrored by Equation \eqref{eqn:wigner} as follows:
\begin{eqnarray}
	d_{m,-s}^{\ell}\left(\pi - \vartheta\right)
	&=& \left( -1\right)^{\ell+s} d_{-m,-s}^{\ell}\left(\vartheta\right)
	,\label{eqn:wignparity}
\end{eqnarray}
see also \cite{vmk}. Furthermore, using the following property 
\begin{equation*}
	\binom{\nu}{\alpha} \Jac{\nu}{-\alpha}{\beta} (t)=\binom{\nu+\beta}{\alpha} (-1)^{\alpha} \left(\frac{1-t}{2}\right)^\alpha  \Jac{\nu-\alpha}{\alpha}{\beta}(t),
\end{equation*}
cf. \cite[Eq. 4.22.2]{sze75}, we can rewrite \eqref{eqn:wigner} as
\begin{align}
	d^{\ell}_{m,-s}\left(\vartheta\right) &= (-1)^{s+m}h_{m}^s(\ell)\left( \sin \frac{\vartheta}{2}\right)^{-m-s}\left( \cos \frac{\vartheta}{2}\right)^{m-s} \Jac{\ell+s}{-m-s}{m-s}\left(\cos\vartheta\right)\notag
\end{align}
\begin{remark}[Jacobi polynomials and other classes of orthonormal polynomials]\label{remark:poly}
	As well-known in the literature (see for example \cite{sze75}), when $\alpha=\beta$ The resulting polynomials are known as Gegenbauer polynomials, or ultraspherical polynomials. These polynomials are connected to the hyperspherical harmonics, which form an orthonormal basis for square-integrable functions defined on the $d$-dimensional sphere. Properties of Gegenbauer polynomials have been previously used in \cite{dp19} to discuss the aliasing effects for $d$ dimensional random fields. When $\alpha=\beta=0$, the resulting polynomials are the Legendre polynomials, which are instrumental in constructing standard 2-dimensional spherical harmonics. Furthermore, the Legendre associated function can be linked to the Jacobi polynomials through the formula:
	\begin{equation}\label{eq:LJ}
		P_{\ell,m}(\cos \vartheta) = \frac{\left(\ell+m\right)!}{2^{m}\ell !} \sin^m \vartheta P_{\ell-m}^{m,m}(\cos \vartheta).
	\end{equation}
	Properties of Legendre polynomials and associated functions have previously been leveraged in \cite{ln97} to explore aliasing effects in $2$-dimensional spherical random fields.
\end{remark}
For any $\ell \geq s$, $m=-\ell,\ldots,\ell$, we can define the spin $s$ harmonic coefficient as 
\begin{eqnarray}\notag
	\alms& =& \int_{0}^{2\pi}\int_{0}^{\pi} f_s(\vartheta,\varphi) \overline{\Ylms}(\vartheta,\varphi) \sin \vartheta \diff \vartheta \diff \varphi\\
	&=& (-1)^{m} \sqrt{\frac{2\ell+1}{4\pi}}  \int_{0}^{2\pi}\int_{0}^{\pi} f_s(\vartheta,\varphi) e^{-im\varphi} d_{-m,s}^{\ell}\left(\vartheta\right)\sin \vartheta \diff \vartheta \diff \varphi \notag.
\end{eqnarray}
Observe that 
\begin{equation}\label{eq:cc}
	\overline{\alms}=-(1)^{s+m}a_{\ell,-m;-s}
\end{equation}
The spin-$s$ function $f_s$ can be then represented by
\begin{align*}
	f_s\left( \vartheta,\varphi\right) & = \sumlm \alms \Ylms\left( \vartheta,\varphi\right)\\
	&=(-1)^{s}\sumlm \sqrt{\frac{2\ell+1}{4\pi}} \alms d^{\ell}_{m,-s}\left(\vartheta\right) e^{im\varphi}\\
	&=(-1)^{s}\sum_{\ell\geq s}  \sqrt{\frac{2\ell+1}{4\pi}}\sum_{m=-\ell}^{\ell} h_{s}^{m}(\ell) \alms  \\
	&\left(\frac{1-\cos\vartheta}{2}\right)^{\frac{m+s}{2}}\left(\frac{1+\cos\vartheta}{2}\right)^{\frac{m-s}{2}}\Jac{\ell-m}{m+s}{m-s}\left(\cos\vartheta\right),
\end{align*}
or, alternatively, 
\begin{align*}
	f_s\left( \vartheta,\varphi\right)&=\sum_{\ell\geq s}  \sqrt{\frac{2\ell+1}{4\pi}}\sum_{m=-\ell}^{\ell}(-1)^m h_{m}^{s}(\ell) \alms \\
	&\left(\frac{1-\cos\vartheta}{2}\right)^{\frac{-m-s}{2}}\left(\frac{1+\cos\vartheta}{2}\right)^{\frac{m-s}{2}}\Jac{\ell+s}{-m-s}{m-s}\left(\cos \vartheta\right).
\end{align*} 
Before concluding this section, we introduce another characterization of spin functions (see, for instance, \cite{cabkam05, gm11} and references therein). It can be demonstrated that there exists a scalar complex-valued function $g(\vartheta,\varphi) = \text{Re}(g)(\vartheta,\varphi)+i\text{Im}(g)(\vartheta,\varphi),$ such that 
\begin{eqnarray*}\notag
	f_s(\vartheta,\varphi)&=&f_E (\vartheta,\varphi)+i f_B (\vartheta,\varphi)\\
	&=& \sumlm \almE \Ylms(\vartheta,\varphi)+ i \sumlm \almB \Ylms(\vartheta,\varphi),
\end{eqnarray*} 
where 
\begin{equation*}
	f_E(\vartheta,\varphi) = \eth^s  \text{Re}(g)(\vartheta,\varphi); \quad f_B(\vartheta,\varphi) = \eth^s  \text{Im}(g)(\vartheta,\varphi),
\end{equation*}
and such that, for any $\ell \geq s$ and $m=-\ell,\ldots,\ell$, 
\begin{align*}
	\alms = \almE+i \almB.
\end{align*}
The labels $E$ and $B$ stem from cosmological literature where they typically denote the electric and magnetic components ($E$ and $B$ modes) of the Cosmic Microwave Background radiation (see \cite{cabkam05}). Straightforward calculations lead to
\begin{eqnarray*}\label{eq:almE}
	\almE &= &\frac{\alms+\overline{a_{\ell,-m;s}}}{2},\\
	\almB &= &\frac{\alms-\overline{a_{\ell,-m;s}}}{2}\label{eq:almB},
\end{eqnarray*}
see also \cite{glm2009,marpec11}.
\subsection{Spin spherical random fields}\label{sec:spinfields}
In this section, we deal with random isotropic spin functions $T_s$. We assume that there exist a probability space $\left\{\Omega, \mathcal{F}, P\right\}$ such that for all choices of charts $U_{\mathcal{R}}$, the ordinary random function $\left(T_s\right)_{\mathcal{R}}$, defined on $\Omega \times \Stwo$ is jointly $ \mathcal{F} \times \mathcal{B}\left( U_{\mathcal{R}} \right)$-measurable, where $\mathcal{B}\left(U_{\mathcal{R}}\right)$ denotes the Borel sigma-algebra on $U_{\mathcal{R}}$. In particular, as for the scalar case, for the spin random function $T_s(x)=\left(Q + iU\right)\left(x\right)$  the following representation holds 
\begin{equation}\label{eq:field}
	T_s (x)= \sumlm \alms \Ylms, 
\end{equation}
where, for any $\ell \geq s$, $m = -\ell, \ldots, \ell$, and $\omega \in \Omega$, $\alms=\alms(\omega)$ is a complex-valued random element, given by
\begin{eqnarray}\alms &=& \int_{\Stwo} T_s\left(x\right)\overline{\Ylms}\left(x\right)\diff x \notag\\
	&=& \int_{0}^{2\pi} \int_{0}^{\pi} T_s\left(\vartheta,\varphi\right)\overline{\Ylms}\left(\vartheta,\varphi\right)\sin \vartheta \diff \vartheta \diff \varphi \label{eq:alms},
\end{eqnarray}
such that the set $\left\{\alms=a_{\ell,m;E}+i a_{\ell,m;B}:\ell \geq s, m=-\ell, \ldots,\ell \right\}$ contains all the stochastic information pertaining to $T_s$ (see \cite{gm11}).
The equality \eqref{eq:field} holds in the $L^2$-sense, that is,
\begin{equation*}
	\lim_{L\rightarrow \infty} \Ex \left[ \int_{\Stwo} \left\vert  T_s(x) - \sum_{\ell=s}^{L}\summ \alms \Ylms\left(x\right) \right \vert^2 \diff x \right] =0,
\end{equation*}
where  the truncation bandwidth $L$ controls the field’s resolution. Also, we can define expectation $\mu_s:\Stwo\mapsto \mathbb{R}$ and covariance function $\Upsilon_s:\Stwo\times\Stwo\mapsto \mathbb{R}$ of the spin random field as 
\begin{align}\label{eq:expect}
	&\mu_s(x) = \Ex \left[T_s(x) \right],\\
	&\Gamma_s(x,x^\prime)=\Ex\left[\left( T_s \left(x\right)- \mu_s\left(x\right)\right) \left( \overline{T_s} \left(x^\prime\right)- \overline{\mu_s}\left(x^\prime\right)\right)\right]. \label{eq:covar}
\end{align}
In view of Equation \eqref{eq:expect}, $\mu_s$ is a spin $s$ function. Without losing any generality, we can make the assumption that $\mu_s(x) = 0$ for all $x \in \Stwo$. Thus, $\Ex\left[ \alms\right] =0$ for $\ell\geq s$, and $m=-\ell,\ldots,\ell$. Regarding to the covariance function, recall that a (scalar) spherical random field $T$ is isotropic if its distribution is invariant under the group of rotations, meaning $T_s(x)\overset{d}{=}T_s(gx)$ for all $x\in \Stwo$ and $g \in SO(3)$, where 
$\overset{d}{=}$ denotes equality in distribution of random fields (see for example \cite{marpec11}). The spin case is somehow different. For a given $\ell \geq s$, we define the angular power spectrum $\Cls$ as the variance of the $\ell$-th component of the spin random field $T_s$, related to the corresponding set of spherical harmonic coefficients $\left\{\alms:m=-\ell,\ldots,\ell\right\}$, i.e.,
\begin{equation}\notag
	\Cls = \frac{1}{4\pi} \sum_{m=-\ell}^{\ell} \Ex\left[\left\vert\alms\right\vert \right]
\end{equation}
Following \cite{glm2009}, we have that 
\begin{align*}
	&	\Ex\left[a_{\ell,m;E}\overline{a_{\ell^\prime,m^\prime;E}}\right] =\ClE \delta_{\ell}^{\ell^\prime}\delta_{m}^{m^\prime}\\
	&	\Ex\left[a_{\ell,m;B}\overline{a_{\ell^\prime,m^\prime;B}}\right] =\ClB \delta_{\ell}^{\ell^\prime}\delta_{m}^{m^\prime}\\
	&	\Ex\left[a_{\ell,m;E}\overline{a_{\ell^\prime,m^\prime;B}}\right] =\Ex\left[a_{\ell,m;B}\overline{a_{\ell^\prime,m^\prime;E}}\right]=0,
\end{align*}	
such that 
\begin{equation*}
	\Ex\left[\alms \overline{a_{\ell^\prime,m^\prime;s}} \right] =\Cls  \delta_{\ell}^{\ell^\prime}\delta_{m}^{m^\prime},
\end{equation*}
where
\begin{equation*}
	\Cls=\left(\ClE + \ClB\right).
\end{equation*}
Now, following \cite{br14,mal11,ps24,stec22}) and labeling by $\left(\varphi_{\psi},\vartheta_{\psi},\delta_{\psi}\right)$ the Euler angles related to $ \psi(x,x^\prime)$ and using Equation \eqref{eq:summation} in Equation \eqref{eq:covar} yields
\begin{align}\notag
	\Gamma_s(x,x^\prime)&=\Gamma_s(\psi(x,x^\prime))\\  
	&=\sum_{\ell \geq s} \frac{2\ell+1}{4\pi}\Cls D_{s,s}^{\ell}\left( \psi(x,x^\prime)\right) \notag\\
	&=\sum_{\ell \geq s} \frac{2\ell+1}{4\pi}\Cls e^{is\left(\varphi_{\psi}+\delta_{\psi}\right)}
	d^{\ell}_{s,s}(\vartheta_\psi)\notag\\
	&= \kappa (\vartheta_\psi)e^{is\left(\varphi_{\psi}+\delta_{\psi}\right)}, \notag
\end{align}
where $\kappa$ is the circular covariance function given by
\begin{align*}
	\kappa (\vartheta_\psi)&=\sum_{\ell \geq s} \frac{2\ell+1}{4\pi}\Cls 
	d^{\ell}_{s,s}(\vartheta_\psi) \\
	&= \sum_{\ell \geq s} \frac{2\ell+1}{4\pi}\Cls\left( \cos \frac{\vartheta_{\psi}}{2}  \right)^{2s} P_{\ell-s}^{0,2s}\left(\cos \vartheta_{\psi}\right)
\end{align*}
using in the last equality a proper modification of Equation \eqref{eqn:wigner}. \\
To ensure that the spin random field admits a finite circular covariance function and well-defined second-order moments, 
we impose a mild parametric constraint on the angular power spectrum $\Cls$.  
\begin{condition}\label{cond:summability}
There exist two constants $G>0$, governing the overall magnitude of the spectrum, and $\alpha \ge 2$, governing its asymptotic rate, such that
\begin{equation}\label{eq:cls}
	\Cls \le G\,(1+\ell)^{-\alpha}.
\end{equation}
\end{condition}
This assumption is classical in the theory of isotropic random fields on the sphere 
(see, for example, \cite{marpec11,glm2009}) and guarantees the finiteness of the total variance,
so that the field has finite energy and finite circular covariance.  
Since $\mathbb{E}[|a_{\ell,m;s}|^2] = \Cls$, the same bound holds for each harmonic coefficient:
\[
\mathbb{E}\!\left[|a_{\ell,m;s}|^2\right] \le G(1+\ell)^{-\alpha}.
\]
The summability implied by~\eqref{eq:cls} ensures that all series involving the squared coefficients are absolutely convergent, 
allowing the interchange of expectation and summation in later analytical developments via Fubini–Tonelli.  
For completeness, we note that condition~\eqref{eq:cls} is imposed separately on the real (electric) and imaginary (magnetic) 
components,
\[
C_{\ell}^{E} = G_E(1+\ell)^{-\alpha}, 
\qquad 
C_{\ell}^{B} = G_B(1+\ell)^{-\alpha},
\]
so that the combined spin-$s$ spectrum $\Cls = C_{\ell}^{E} + C_{\ell}^{B}$ satisfies the same decay rate with $G=G_E+G_B$.\\
We conclude this section, recalling the concept of band limited spin random fields. As in the scalar case (see for example \cite{dp19,ln97}), a random field is band-limited if there's a bandwidth $L_0 \in \mathbb{N}$, 
such that $\alms =0$ whenever $\ell \geq L_0$ for any $m=-\ell, \ldots,\ell$. In this scenario, \eqref{eq:field} becomes 
\begin{equation}\notag
	T_{L_0,s} (\vartheta,\varphi)= \sum_{\ell =s}^{L_0} \sum_{m=-\ell}^{\ell} \alms \Ylms(\vartheta,\varphi).
\end{equation}
From a practical standpoint, band-limited random fields offer a convenient approximation for fields where the harmonic coefficients decay rapidly with increasing frequency. In general, bandlimiting allows for efficient representation and compression of this signals by focusing on the most significant modes and reducing the computational and storage requirements (see for instance \cite{MW11}), Additionally, bandlimiting holds significant theoretical interest within the field of probability theory (see, for example, \cite{todino23}).

\section{Spherical sampling methods for spin objects}\label{sec:unif}
Let us consider the spin $s$ field $T_s$. Sampling $T_s$ consists in the process of converting the continuous field into a discrete set of values measured over the sampling points $\left\{x_k=\left(\vartheta_k,\varphi_k\right):k=1,\ldots,K\right\}$, to obtain the samples of the random field $\left\{T_s\left(x_k\right)=\left(\vartheta_k,\varphi_k\right):k=1,\ldots,K\right\}$. In general, this procedure involves sampling at points distributed according to a specific scheme that aligns with the mathematical structure of the chosen type of studied random field. For instance, combinations of the trapezoidal rule and Gauss-Legendre or Gauss-Gegenbauer formulas are typically used for scalar spherical random fields on $\Stwo$ or $\mathbb{S}^d$, respectively (see \cite{ln97,dp19}).\\
In this section, we first introduce a very mild condition on the sampling procedure, that is, assuming that the sampling grid is built by combining two one-dimensional quadrature rules, one for the colatitude $\vartheta$ and the other one for the longitude $\varphi$. In general, a quadrature rule serves as a technique for estimating the definite integral of a function. It achieves this by sampling the function at specific points across the integration domain and then aggregating these sampled values through weighted summation (see, for example, \cite{stobur93}).\\
We then present two of the most popular spherical sampling methods for spin objects, highlighting some similarities to establish analogous locations of the aliases. Given the choice of the sampling scheme, those aliases will have different relevance within the costruction of the aliased coefficient. These sampling schemes differ in the sampling strategy for the colatitude $\vartheta$, while the trapezoidal rule is used for the angular coordinate $\varphi$, providing a straightforward approach for integrating over this coordinate, and a specific scheme for the angle $\vartheta$. 
Indeed, the trapezoidal rule is a numerical integration method used to approximate the definite integral of a function. It operates by approximating the area under the curve of the function by dividing the interval into small segments and approximating each segment by a trapezoid (see again \cite{stobur93}).
\subsection{Separability of the sampling scheme}\label{sec:separab}
Extending the proposal introduced by \cite{ln97} for scalar functions on the sphere to the spin case, we consider a discretization scheme produced by the combination of two one-dimensional quadrature rules, with respect to the coordinates $\vartheta$ and $\varphi$ to  create a grid of points by pairing each colatitude angle with each longitude angle.
Separating the coordinates is a general successful strategy to reduce complex problems on the sphere to simpler, more manageable ones. Furthermore, this separation can simplify both the analysis and the computation of spin and scalar random fields defined on the sphere. For example, it allows the use of Fast Fourier Transform (FFT) algorithms on the sphere (see for example \cite{HW10,MW11}), enhancing computational efficiency, especially for large-scale data sets encountered in cosmology and geophysics. Furthermore, this separation aligns well with the mathematical framework of spherical harmonics. As discussed in Section \ref{sec:back}, the standard representation of spherical objects consists in their harmonic expansion, which naturally separate variables in spherical coordinates.
\begin{condition}\label{cond:sep} 
Consider the fixed natural values $N$ and $Q$ such that $K=2NQ$. Then there exist two sequences of sampling points $\left\{ \vartheta_p:p=0,\ldots,N-1 \right\}$, and $\left\{ \vartheta_q: q=0,\ldots,2Q-1 \right\}$, and the corresponding weights $\left\{ w_p^{\left(\vartheta\right)}:p=0,\ldots,N-1\right\}$ and $\left\{ w_q^{\left(\varphi\right)}:q=0,\ldots,2Q-1\right\}$ such that
	\begin{align*}	%
		\sum_{p=0}^{N-1}w_p^{\left(\vartheta\right)}=1;	\quad \sum_{q=0}^{2Q-1}w_q^{\left(\varphi\right)}=1.
	\end{align*}
	The sampling points $\left\{x_k:k=1,\ldots,K \right\}$ are defined component-wise by 
	$$\left\{ \left(\vartheta_p,\vartheta_q\right) :p=0,\ldots,N-1; q=0,\ldots,2Q-1 \right\}.$$ Likewise, each fixed value $\left(p^\ast,q^\ast\right) $ results to a given weight $k^{\ast}=p^\ast q^\ast$, $k^\ast \in \left\{1,\ldots,K \right\}$.
\end{condition}
In the remainder of this section, we will discuss two primary sampling schemes based on coordinate separation. First, we will present the Gauss-Jacobi scheme, which uses the minimum number of sampling points needed to reconstruct a polynomial function of a fixed degree. Next, we will introduce the equiangular scheme, which requires more points, anyway easier to be identified. Finally, we will characterize certain symmetry properties for both schemes, which will be crucial for establishing our main results in the following section.
\subsection{Spin spherical uniform sampling}
The first scheme is the so called spin spherical uniform sampling, defined as a proper adaptation of the scalar spherical uniform sampling method (see \cite{dh2003, dp19, ln97}). This sampling scheme involves a Gauss-Jacobi  (or, more specifically, a Gauss-Gegenbauer) quadrature formula for $\vartheta$.  The Gauss-Legendre quadrature optimally handles the weighting and integration over $\vartheta$, resulting in a uniform and precise sampling method suitable for complex spin functions on the sphere. However, the order of exact reconstruction varies with the spin value $s$.\\
A $\nu$-point Gaussian quadrature is designed to compute exactly integrals of polynomials with degrees up to $2\nu-1$, ensuring accuracy through strategic selection of points and corresponding weights $\left\{ t_k, \omega_k: k=0,\ldots,\nu-1\right\}$ (for further insights, refer to \cite{stobur93}). 
The domain of integration is conventionally taken as $\left[-1,1\right]$; the choice of quadrature points and weights depends on the weight function $W:\left[-1,1\right] \mapsto \mathbb{R}$, so that the integral to be approximated is given by $\int_{-1}^{1}P\left(t\right)W\left(t\right) \diff t$, where $P:\left[-1,1\right]$ is approximately polynomial and $W\in L^1\left[-1,1\right]$ is a well-known function. A quadrature rule of order $2\nu-1$ uses a suitable choice of the set $\left\{ t_k, \omega_k: k=0,\ldots,\nu-1\right\}$ to obtain
\begin{equation*}
	\int_{-1}^{1} P\left(t\right)W\left(t\right) \diff t = \sum_{k=0}^{\nu-1} \omega_k P\left(t_k\right),
\end{equation*}
if $P$ is of degree smaller or equal to $2\nu-1$.\\
Here we consider $W\left(t\right)=\left(1-t\right)^{\alpha}\left(1+t\right)^{\beta}$, $\alpha,\beta>-1$. In this case, where $t_k$ are the nodes and $\omega_k$
are the weights associated with the Gauss-Jacobi quadrature rule. These nodes are determined as the roots of the Jacobi polynomials of degree $\nu$ and weights are given by
\begin{align*}
	\omega_k & = \frac{G_{\nu}^{\alpha,\beta}}{J_{\nu-1}^{\alpha,\beta}(t_k)\frac{d J_{\nu}^{\alpha,\beta}(t_k)}{dt}} \\
	& = \frac{\tilde{G}_{\nu}^{\alpha,\beta}}{\left(1 - t_k^2 \right)\left(\frac{d J_{\nu}^{\alpha,\beta}(t_k)}{dt}\right)^2}, 
\end{align*}
where 
\begin{align*}
	&G_{\nu}^{\alpha,\beta} =  \frac{2^{\alpha+\beta}\left( 2\nu + \alpha + \beta \right)\Gamma (\nu +\alpha) \Gamma(\nu+\beta)}{\nu!\Gamma(\nu+\alpha+\beta+1)}\\
	&\tilde{G}_{\nu}^{\alpha,\beta} = \frac{2^{\alpha+\beta+1}\Gamma (\nu +\alpha+1) \Gamma(\nu+\beta+1)}{\nu!\Gamma(\nu+\alpha+\beta+1)} ,
\end{align*}
see \cite{stobur93}. 
\begin{remark}[Other types of Gauss quadrature rules]\label{rem:gaussformula}
	Observe that for $\alpha=\beta$ we reduce to the Gauss-Gegenbauer quadrature rule, which is utilized to evaluate the aliasing effect for $d$-dimensional random fields as described in \cite{dp19}. On the other hand, setting $\alpha=\beta=0$ results in the Gauss-Legendre formula, which is employed to assess aliasing in the case where $d=2$ as discussed in \cite{ln97}.\\
	The Gauss-Jacobi quadrature is particularly useful in numerical integration when dealing with integrals of the form:
	\begin{equation*}
		\int_{-1}^{1} \left(1-t\right)^{\alpha}\left(1+t\right)^{\beta} g(t) \diff t,
	\end{equation*}
	where $g:[-1,1]\mapsto \mathbb{R}$. When the integrand $g$ is multiplied by a non-uniform weight function $\left(1-t\right)^{\alpha}\left(1+t\right)^{\beta}$,
	typically simpler quadrature rules like Gauss-Legendre might not perform well. Gauss-Jacobi quadrature, which is tailored for such weight functions, can handle these cases more efficiently.
\end{remark}
We now introduce the spin spherical uniform sampling scheme in detail. This method extends the scalar spherical sampling approach by incorporating spin functions, ensuring a comprehensive and uniform sampling across the sphere. The scheme is developed by combining two well-established one-dimensional quadrature rules: the trapezoidal rule for the angular coordinate $\varphi$ and the Gauss-Jacobi quadrature formula for the angular coordinate $\vartheta$. This approach ensures that the sampling points are distributed uniformly and that the integration is performed efficiently, leveraging the strengths of both the trapezoidal and Gauss-Jacobi quadrature rules.\\
In general, Gauss type sampling schemes are typically used to optimize the experimental design on the sphere (see \cite{dette05,dette19}), improving the accuracy of integration or approximation of functions over a given domain once the experimenter has determined the number of samples that will compose the grid.
\begin{condition}[\textbf{Spin spherical uniform sampling}]\label{cond:2}
	Under the assumptions of Condition \ref{cond:sep}, the sampling with respect to $\varphi$ is uniform, so that for $q=0,\ldots,2Q-1$, 
	\begin{equation} \notag 
		\begin{cases}
			\varphi_q = q \frac{2\pi}{2Q} = q \frac{\pi}{Q}\\
			w_q^{\left(\varphi\right)} = \frac{2\pi}{2Q} = \frac{\pi}{Q}.
		\end{cases}
	\end{equation}
	The sampling with respect to $\vartheta$ follows a proper adaptation of the Gauss-Jacobi quadrature rule. Indeed, for $p=0,\ldots,N-1$ 
	\begin{equation}\notag
		\begin{cases}
			\vartheta_p = \arccos t_{p}\\
			w_p^{\left(\vartheta\right)} = \frac{\omega_p}{\sin \vartheta_p},
		\end{cases}
	\end{equation}
	where, $\left\{t_{p}:p=0,\ldots, N-1\right\}$ are the zeros of the Jacobi polynomial $J_{N}^{s,s}$   
\end{condition}
	Consider the quadrature formula for $\vartheta$. Since $\alpha=\beta=s$, as noted in Remark \ref{rem:gaussformula}, this formula reduces to a Gauss-Gegenbauer quadrature rule. Additionally, according to Equation \eqref{eq:LJ}, the zeros of $\Jac{N}{s}{s}$ coincide with those of the Legendre polynomial $P_{N-s}$. This establishes a connection between the quadrature rules for scalar and spin functions, allowing us to use the same quadrature points, even though the integration is exact to a higher degree in the scalar case.
\subsection{The equiangular sampling approach}
Introduced in the literature by \cite{dh2003}, the equiangular sampling involves placing sample points at equal angular intervals in both colatitude $\vartheta$ and longitude $\varphi$ directions, combining two trapezoidal rules. This approach ensures that the points are sampled over equally spaced intervals with respect to each angular coordinate. Even if this approach is easy to implement due to the uniform spacing of sampling points,
it can lead to oversampling in some regions, especially near the poles where the density of points increases.
\begin{condition}[\textbf{Equiangular sampling}]\label{cond:3}
	Under the assumptions of Condition \ref{cond:sep}, the samplings with respect to $\varphi$ and $\vartheta$ are uniform. For $q=0,\ldots,2Q-1$, 
	\begin{equation} \notag 
		\begin{cases}
			\varphi_q = q \frac{2\pi}{2Q+1} \\
			w_q^{\left(\varphi\right)} = \frac{2\pi}{2Q+1}.
		\end{cases}
	\end{equation}
	Let $N$ be such that $N-s=2z$, $z\in \mathbb{N}$, and $N^\prime = 2(N-s)+1$. For $p=0,\ldots,2N^\prime-1$ 
	\begin{equation} \notag \label{eq:thetasampling2}
		\begin{cases}
			\vartheta_p = \frac{\pi p}{2N^\prime}\\
			w_p^{\left(\vartheta\right)} = \frac{2}{N^\prime}\sin\vartheta_p\sum_{n=0}^{N^\prime-1} \frac{\sin((2n+1)\vartheta_p)}{2n+1}.
		\end{cases}
	\end{equation}
\end{condition}
Notice that we introduced the notation $N^\prime$ to indicate that in order to have an exact quadrature formula we require a double amount of samples for the colatitude with respect to the Gauss-Jacobi method.
\begin{remark}[Alternative equiangular sampling schemes]
	An interesting extension of this approach is derived through the factorization of rotations and the periodic extension of the angle $\vartheta$,  allowing for the exploitation of the orthogonality of complex exponentials over the interval $\left[\left.0,2\pi\right)\right.$. Among several alternative constructions based on this approach, here, we describe the construction presented by \cite{MW11}, while we suggest \cite{HW10, wg01} among the alternative constructions. As shown by \cite{risbo}, a Wigner's small $d$ function can be represented by the following Fourier series decomposition,
	\begin{equation*}
		d^{\ell}_{m,-s}=i^{-s-m}\sum_{n=-\ell}^{\ell}\Delta_{n,m}^{\ell}\Delta_{n,-s}^{\ell}e^{in\vartheta},
	\end{equation*}
	where we use the same notation introduced in \cite{HW10,MW11},
	\begin{equation*}
		\Delta^\ell_{n,m}=d_{n,m}^{\ell}\left(\frac{\pi}{2}\right).
	\end{equation*}
	It follows that Equation \eqref{eq:alms} can be rewritten as
	\begin{equation*}
		\alms=(i)^{s-m} \sqrt{\frac{2\ell+1}{4\pi}}\sum_{n=-\ell}^{\ell} 	\Delta^\ell_{n,m}	\Delta^\ell_{n,-s}\int_{0}^{\pi} \int_{0}^{2\pi}   T_s\left(\vartheta,\varphi\right) e^{-im\varphi} \sin\vartheta  e^{-in\vartheta}\diff \varphi \diff \vartheta, 
	\end{equation*}
	Using the orthogonality of the complex exponential basis (in both the discrete and continuous versions) yields to sampling scheme for the colatitude $\vartheta$. 
	\begin{equation} \notag \label{eq:thetasampling3}
		\begin{cases}
			\vartheta_p = \frac{\pi(2p+1) }{2N^\prime-1}, \quad p=0,\ldots N^\prime-1 \\
			w_p^{\left(\vartheta\right)} = \omega(\vartheta_p)+(1-\delta_p)^{2N^\prime-1}\vartheta_{2L-2-p},
		\end{cases}
	\end{equation}
	where, $\omega(\vartheta_p)$ is the inverse discrete Fourier transform of the rescaled weight function
	\begin{equation*}
		W(n)=\int_{0}^{2\pi} e^{in\vartheta}\sin \vartheta \diff \vartheta = \begin{cases}
			\pm \frac{i \pi}{2} & n=\pm 1\\
			0 &  n \text{ odd, }n\neq \pm 1\\
			\frac{2}{(1-n^2)} &n  \text{ even }
		\end{cases}.
	\end{equation*}
	This approach presents two main advantages. On the one hand, the Fourier series representation
	of the Wigner $d$ function allows the spherical harmonic expansion of $T_s\left(\vartheta,\varphi\right)$, after a proper extension, to be expressed as a Fourier series expansion on the two-dimensional torus. On the other hand, all rotations can be represented using the Wigner $d$-functions evaluated only at $\pi/2$, which are straightforward to compute (see, for example, \cite{risbo}). An advantage of this approach is that it requires less than half the number of samples compared to the canonical equiangular sampling theorem (see \cite{MW11}.) However, in what follows, we will refer to the original scheme for simplicity in constructing the weight function for $\vartheta$. The same results can be extended to this more refined approach, provided the symmetry of the sampling points for the colatitude is ensured.\\
	We conclude by noting that several widely used sampling strategies in the applied literature are not based on separable tensor-product constructions. A prominent example is the HEALPix scheme \cite{gor05}, which relies on a hierarchical, iso-latitude pixelization optimized for fast spherical harmonic transforms rather than exact quadrature separability. Such approaches fall outside the framework considered in this section, but they remain central in large-scale cosmological data analysis and related applications.
\end{remark}

\subsection{The parity of quadrature points}\label{sec:parity} 
The Gauss-Jacobi sampling scheme described in this section has a very regular distribution of quadrature points. While this consideration is obvious for the longitude $\varphi$, where we entail a trapezoidal rule, as far as the colatitude $\varphi$ is concerned, all the samplings enjoy a symmetry with respect to $\pi/2$. It follows the parity relation \eqref{eq:parity}, which yields 
\begin{equation*}
	\Jac{N}{s}{s}(-t)=(-1)^{N}\Jac{N}{s}{s} (t), \quad \text{for }t\in[-1,1].
\end{equation*}
Thus also in this case, the quadrature points can be arranged symmetrically, 
$$t_p = t_{N-1-p}, \text{ for }p=0, \ldots,\left[\frac{N}{2}\right],$$ 
similar also to the scalar case. \\
As far as the equiangular scheme is concerned, this seems to be no more true. Indeed, while for $p=1, \ldots,N^\prime-1$, the cubature points keep the same symmetry described above for the Gauss-Jacobi method, for $p=0$ we obtain $\vartheta_0= 0$, which breaks this symmetry. Anyway, this problem is solved by the practical point of view by means of the weight function, always null for this point in view of his structure.  
As a consequence, for each harmonic coefficient $\alms$, both the sampling schemes will share the same locations for the aliases, though they will differ in terms of intensity.\\ 
In the scalar case, this fact has even a stronger consequence, since it imply that also for the equiangular and toroidal methods, for a fixed multipole $\ell$ (odd or even), this symmetry results in the annihilation of all even or odd aliases respectively (see \cite{dp19,ln97}). However, this property does not apply in general here but in a specific case, as discussed below in Remark \ref{rem:comp}.     
\section{Measuring aliasing effects for spin random fields}\label{sec:alias}
In this section, we first introduce the aliasing function for separable spin sampling schemes, along with basic concepts such as the locations of aliases and their distances from the aliased coefficient. We then elaborate on the locations and intensities of the aliases, assuming the sampling follows one of the schemes described in the previous section. Subsequently, we establish criteria for eliminating the deletable aliases. We then examine the aliasing structure affecting the angular power spectrum of the random field. Finally, we determine the conditions under which bandlimited random fields are free from aliases.
\subsection{Introducing the spin aliasing function}
Measuring aliasing effects for spin random fields involves evaluating how high-frequency components of the fields are misrepresented when sampled at a discrete set of points, potentially leading to inaccuracies in the decomposition and reconstruction of these fields. The phenomenon of aliasing, or confounding, occurs when high-frequency components of the field are undersampled, causing them to be indistinguishable from lower-frequency components, leading to distortions in the reconstructed field. More formally, when measuring the random field over a finite number of locations,  the harmonic coefficients  $\left\{\alms:\ell \geq s, m=-\ell,\ldots,\ell\right\}$ cannot be computed explicitly via continuous integration. Instead, the integral for computing these coefficients must be approximated by a sum over a discrete set of samples of $T_s$, which are measurements of the field at specific sampling points. As a consequence, for any $\ell \geq s$, $m=-\ell,\ldots,\ell$, we can evaluate the aliased coefficient as follows 
\begin{eqnarray}
	\almstilde & = & \sum_{k=1}^{K} w_k T_s\left(\vartheta_k,\varphi_k\right) \overline{\Ylms}\left(\vartheta_k,\varphi_k\right) \sin \vartheta_k \notag \\
	& = &  \sum_{k=1}^{K} w_k \left( \sum_{u \geq s} \sum_{v=-u}^{u} a_{u,v;s}  Y_{u,v;s}\left(\vartheta_k,\varphi_k\right) \right) \overline{\Ylms}\left(\vartheta_k,\varphi_k\right) \sin \vartheta_k \notag\\
	& = & \sum_{u \geq s}\sum_{v=-u}^{u}\Amat  a_{u,v;s}, \label{eq:alias1} 
\end{eqnarray}
where
\begin{equation}\notag
	\Amat =  \sum_{k=1}^{K} w_k    Y_{u,v;s}\left(\vartheta_k,\varphi_k\right)\overline{\Ylms}\left(\vartheta_k,\varphi_k\right) \sin \vartheta_k
\end{equation}
is the aliasing function or matrix.
Note that, in view of Equation \eqref{eq:cc} 
\begin{eqnarray}
	\overline{\almstilde} 
	& = &\sum_{u \geq s}\sum_{v=-u}^{u}  (-1)^{s+v}\tau_s\left(u,v;\ell,m\right)a_{u,-v;-s}\notag
\end{eqnarray}
\begin{remark}[Locations, intensity and distance of the aliases]\label{rem:loc}
	Recall that $a_{u,v;s}$ is an alias of $\alms$ with intensity $\left \vert \Amat\right \vert$ if $\Amat \neq 0$.\\
	Following \cite{dp19,ln97}, the location in the frequency domain of the harmonic coefficient $\alms$ is identified by the the couple $(\ell,m)$, so that the distance in the frequency domain of the alias $a_{u,v;s}$ from $\alms$ is given by the Euclidean $\ell_2$ distance of the numbers $(\ell,m)$ and $(u,v)$ in the space of square-summable sequences. 
	\begin{equation*}
		\text{dist}\left( \alms,a_{u,v;s}\right) = \sqrt{(u-\ell)^2+(v-m)^2}.
	\end{equation*}
\end{remark}
We can express the aliasing function in terms of the Wigner's $d$ matrices
\begin{align*}\label{eq:Amatrixwig}
	\Amat &=  (-1)^{s+m}\sum_{k=1}^{K} w_k  \frac{2\ell+1}{4\pi} e^{-i(m-v)\varphi_k} d^{u}_{v,-s}(\vartheta_k)d^{\ell}_{-m,s}(\vartheta_k)\sin \vartheta_k\\
	&= \frac{2\ell+1}{4\pi}  \sum_{k=1}^{K} w_k e^{-i(m-v)\varphi_k} d^{u}_{v,-s}(\vartheta_k)d^{\ell}_{m,-s}(\vartheta_k)\sin \vartheta_k,
\end{align*}
and
\begin{eqnarray*}
	\overline{\Amat} &=& 
	\tau_s\left(u,v;\ell,m\right)
\end{eqnarray*}
Let us now focus on the sets of coefficients $\left\{\almE \right\}$ and $\left\{\almB \right\}$. It follows from \eqref{eq:cc} that their aliased versions are given by
\begin{align*}
	\tilde{a}_{\ell,m;E} & =  \frac{\almstilde+ \overline{\tilde{a}_{\ell,-m;s}}}{2}\\
	& = \sum_{u \geq s}\sum_{v=-u}^{u}\frac{\left( \tau_s(\ell,m;u,v)a_{u,v;s} + \tau_s(u,v;\ell,-m)\overline {a_{u,v;s}} \right) }{2}\\
	& = \sum_{u \geq s}\sum_{v=-u}^{u}\frac{\left( \tau_s(\ell,m;u,v)a_{u,v;s} +(-1)^{s+v} \tau_s(u,v;\ell,-m){a_{u,-v;-s}} \right) }{2} ; 
\end{align*}
\begin{align*}
	\tilde{a}_{\ell,m;B} & =  \frac{\almstilde- \tilde{a}_{\ell,-m;s}}{2}\\
	& = \sum_{u \geq s}\sum_{v=-u}^{u}\frac{\left( \tau_s(\ell,m;u,v)a_{u,v;s} - \tau_s(u,v;\ell,-m)\overline {a_{u,v;s}} \right) }{2}\\
	& = \sum_{u \geq s}\sum_{v=-u}^{u}\frac{\left( \tau_s(\ell,m;u,v)a_{u,v;s} -(-1)^{s+v} \tau_s(u,v;\ell,-m){a_{u,-v;-s}} \right) }{2}.
\end{align*}

Under Condition \ref{cond:sep}, the aliasing function becomes component-wise separable, as shown in the following Theorem \ref{th:sep}, whose proof is given in Section \ref{sec:proof}.
As a result of the separability of the coordinates, the aliasing function also becomes separable. The following theorem highlights how each component of the aliasing function depends on the locations in the frequency domain of both the aliased coefficient and its corresponding aliases.
\begin{theorem}\label{th:sep}
	Let us assume that Condition \ref{cond:sep} holds. Then it follows that
	\begin{eqnarray*}
		\Amat & = & \frac{\sqrt{\left(2\ell + 1\right)\left(2 u + 1 \right)}}{4\pi}  I_{N}\left( \ell, m; u, v\right) H_Q\left(m,v\right), 
	\end{eqnarray*}
	where 
	\begin{equation}\label{eq:IP}
		I_{N}\left( \ell, m; u, v\right) =  \sum_{p=0}^{N-1} w_p^{\left(\vartheta\right)} 
		d^{\ell}_{m,-s}\left(\vartheta_p\right)d^{u}_{v,-s}\left(\vartheta_p\right) \sin \vartheta_p,\\
	\end{equation}
and
\begin{equation}
	 H_Q\left(m,v\right) = \sum_{q=0}^{2Q-1}  w_q^{\left(\varphi\right)} e^{i\left(v-m\right)\varphi_q}.\label{eq:HQ}
\end{equation}
\end{theorem}
This separable form is independent of the specific quadrature rule and holds for
	all sampling schemes considered later.  
	Accordingly, all numerical examples in Section~5 rely on this same separable
	construction; only the underlying one–dimensional rules for $\vartheta$ and
	$\varphi$ change from case to case.\\
The next section will explore in greater detail how to classify aliases within a sampling scheme and outline the strategies for removing those that are deletable.
\subsection{Aliasing for separable symmetric schemes}
Let us introduce some notation that will be useful for presenting our main result.
For any $ \ell \geq s$ and $m=-\ell,\ldots,\ell$, let us define the following sets
\begin{align}
	& R_{0,m}^{Q}\left(u\right) = \left\{r:\left\vert m + 2rQ \right\vert <u  \right\} = \left\{ -\frac{u+m}{2Q} \leq r \leq \frac{u-m}{2Q}  \right\},\label{eq:Rm}\\
	&R_{m}^{Q}\left(u\right) = \left\{r:\left\vert m + 2rQ \right\vert <u  \right\} = \left\{ -\frac{u+m}{2Q} \leq r \leq \frac{u-m}{2Q}, r \neq 0 \right\}, \notag 
\end{align}
while 
\begin{align*}
	&A_{\ell,s;N}=\left\{j\in\mathbb{Z}: -\ell \leq j \leq N-s-1  \right\},\\
	&B_{s;N}=\left\{j\in\mathbb{Z}: j > N-s-1  \right\},
\end{align*}
so that $D_\ell =[-\ell,\infty)$ corresponds to $A_{\ell,s;N} \cup B_{s;N}$. Note that for $\ell > N$, $A_{\ell,s;N}=\emptyset$ and $D_\ell= B_{s;N}$. Similar to \cite{dp19}, we will denote the aliases identified by 
$j\in B_{s;N}$ as \textit{primary aliases}, while those with $j \in A_{\ell,s;N}$ will be referred to as \textit{secondary aliases}.
For any $\ell \geq s$, $m=-\ell,\ldots,\ell$, the next results will identify and classify the location of the aliases for the harmonic coefficient $\alms$, under the assumptions of a trapezoidal rule for $\varphi$ and the parity relation among sampling points for the colatitude $\varphi$ presented in Section \ref{sec:parity}. Classifying aliases into primary and secondary categories will subsequently enable the development of strategies to remove the secondary aliases.
The proof is included in Section \ref{sec:proof}. 
\begin{theorem}\label{th:main}
	Assume that the random field $T_s$ is sampled by one of the schemes described by Conditions \ref{cond:2} or \ref{cond:3}, and for any $z_1,z_2 \in \mathbb{N}$, let 
	\begin{equation*}
		\kappa_{z_1,z_2} = \frac{\sqrt{(2z_1+1)(2z_2+1)}}{2}.
	\end{equation*} 
For any $\ell \geq s$ and $m=-\ell, \ldots, \ell$, and for any $r^\prime \in \mathbb{Z}$, it holds that 	\begin{align*}\notag
		H_Q\left(m,v\right) & = \begin{cases} 2\pi & \text{for }v=m+2r^\prime Q \\
			0 & \text{otherwise}.
		\end{cases}, \notag 
	\end{align*}
	Furthermore, for $m\neq 0$, the aliased coefficient $\almstilde$ defined by Equation \eqref{eq:alias1} is given by
	\begin{align*}
		\almstilde & = \sum_{j\in A_{\ell,s;N}} \sum_{r \in R_{m}^{Q}  \left(\ell+j\right)} \kappa_{\ell,\ell+j}  I_{N}(\ell,m,\ell+j,m+2rQ) a_{\ell+j,m+2rQ;s}\\
		& + \sum_{j\in B_{s;N}} \sum_{r \in R_{0,m}^{Q}  \left(\ell+j\right)}   \kappa_{\ell,\ell+2j}  I_{N}(\ell,m,\ell+j,m+2rQ) a_{\ell+j,m+2rQ;s},
	\end{align*}
	while for $m=0$  
	\begin{align*}
		\tilde{a}_{\ell,0;s} & = \sum_{j\in A_{\ell,s;N}} \sum_{r \in R_{m}^{Q}  \left(\ell+j\right)}   \kappa_{\ell,\ell+j} I_{N}(\ell,m,\ell+j,m+2rQ) a_{\ell+j,m+2rQ;s}\\
		& + \sum_{j\in B_{s;N}} \sum_{r \in R_{0}^{Q}  \left(\ell+j\right)}    \kappa_{\ell,\ell+j}  I_{N}(\ell,0,\ell+j,2rQ) a_{\ell+j,2rQ;s}\\
		& + \sum_{j\in B_{s;N}} \kappa_{\ell,\ell+2j} I_{N}(\ell,0,\ell+2j,0) a_{\ell+2j,0;s}.
	\end{align*}
\end{theorem}
For a clearer understanding of these theoretical results, Section~5
illustrates their concrete behaviour through the aliased coefficient
$\tilde a_{2,0;2}$ computed under both sampling schemes  with $N=6$. In the Case~I we will choose $Q=1$, while Case~II has $Q=2$.  The corresponding aliasing weights and numerical values of the aliasing function are reported in Tables~\ref{tab:table1} and \ref{tab:table2} respectively.
\begin{remark}[Comparison with the scalar case]\label{rem:comp}
	In Section \ref{sec:parity} we noted that the cubature points in this sampling scheme coincide with those developed for the scalar case in \cite{ln97}, even though the integration here is exact up to a lower degree of the spin spherical random fields ($N-s$ versus $N$).  There is another crucial difference between these two cases. In the scalar case, in both dimension 2 and higher dimensions $d>2$ (cf. \cite{dp19}), for a fixed $\ell \in \mathbb{N}$, combining Equations \eqref{eq:parity} and \eqref{eq:LJ} and considering the symmetry of the cubature points, it follows that the aliases of each $a_{\ell,m}$ are characterized by a multipole with the same parity as $\ell$. That is, coefficients with even (odd) $\ell$ have aliases only with even (odd) multipoles. This property does not hold for spin $s\neq 0$, except in the case where $m=0$ and $r=0$.  A similar symmetry is preserved for any $\ell\geq  $ and $m=-\ell,\ldots,\ell$, by fixing $r=-\frac{m}{Q}$ and $j=0$. Indeed, in this case Equation \eqref{eq:IP2} becomes
	\begin{align*}
		I_{N}\left( \ell, m; \ell, -m\right) &=  \sum_{p=0}^{N-1} w_p^{\left(\vartheta\right)} 
		d^{\ell}_{m,-s}\left(\vartheta_p\right)d^{\ell}_{-m,-s}\left(\vartheta_p\right) \sin \vartheta_p.
	\end{align*}
	Following Equation \eqref{eqn:wignparity} yields
	$$
	d^{\ell}_{m,-s}\left(\pi-\vartheta\right)d^{\ell}_{-m,-s}\left(\pi-\vartheta\right) \sin\left(\pi-\vartheta\right)= (-1)^{2(\ell+s)}d^{\ell}_{-m,-s}\left(\vartheta\right)d^{\ell}_{m,-s}\left(\vartheta\right) \sin\left(\vartheta\right).
	$$
	so that $ I_{N}\left( \ell, m; \ell, -m\right)$ can be evaluated using only half of the cubature points.
\end{remark}
The next result will show that a proper choice of the number of sampling points can eliminate the secondary aliases. However, it is impossible to remove also all the primary aliases solely by choosing sampling points and parameters. As discussed below, these aliases can only be all completely removed if the random field is band-limited and the sampling points are properly selected. The proof of this result is in Section \ref{sec:proof}.
\begin{corollary}\label{cor:1}
	In the same setting proposed within Theorem \ref{th:main}, assume that $Q>N-s$. Then, for $m \neq 0$, it holds that
	\begin{align*}
		\almstilde & =  \sum_{j\in B_{s;N}} \sum_{r \in R_{0,m}^{Q}  \left(\ell+j\right)}    \kappa_{\ell,\ell+j} I_{N}(\ell,m,\ell+j,m+2rQ) a_{\ell+j,m+2rQ;s},
	\end{align*}
	while for $m=0$  
	\begin{align*}
		\tilde{a}_{\ell,0;s} & =\sum_{j\in B_{s;N}} \left(\sum_{r \in R_{0}^{Q}  \left(\ell+j\right)}   \kappa_{\ell,\ell+j} I_{N}(\ell,0,\ell+j,2rQ) a_{\ell+j,2rQ;s}\right.\\
		& + \left. \kappa_{\ell,\ell+2j} I_{N}(\ell,0,\ell+2j,0) a_{\ell+2j,0;s}\right).
	\end{align*}
	Moreover, the minimum distance between the aliased coefficient and its aliases is given by
		\begin{align*}
		N\sqrt{5+\frac{s^2}{N^2}-\frac{2s}{N}}
	\end{align*}
\end{corollary}
Observe that under the hypotheses assumed in Corollary \ref{cor:1}, the electric and the magnetic harmonic coefficients become
\begin{align*}
	\tilde{a}_{\ell,m;E}& =\sum_{j\in B_{s;N}} \frac{ \kappa_{\ell,\ell+j}}{2}\left[\sum_{r\in R_{m}^{Q}(\ell+j)}  I_{N}(\ell,m;\ell+j,m+2rQ) a_{\ell+j,m+2rQ;s}\right.
	\\ &\left. +\sum_{r\in R_{-m}^{Q}(\ell+j)}(-1)^{s-m+2rQ} I_{N}(\ell+j,-m+2rQ;\ell,-m)a_{\ell+j,m-2rQ;-s}\right] ; 
\end{align*}
\begin{align*}
	\tilde{a}_{\ell,m;B}& =\sum_{j\in B_{s;N}} \frac{ \kappa_{\ell,\ell+j}}{2}\left[\sum_{r\in R_{m}^{Q}(\ell+j)}  I_{N}(\ell,m;\ell+j,m+2rQ) a_{\ell+j,m+2rQ;s}\right.
	\\ &\left. -\sum_{r\in R_{-m}^{Q}(\ell+j)}(-1)^{s-m+2rQ} I_{N}(\ell+j,-m+2rQ;\ell,-m)a_{\ell+j,m-2rQ;-s}\right].
\end{align*} 
Replacing the integral \eqref{eq:alms} with the sum \eqref{eq:alias1} under the assumptions given in Condition \ref{cond:2}, we are now interested in understanding how the aliasing phenomenon impacts the estimation of the angular power spectrum $\Cls$, similar to the scalar case discussed in \cite{ln97}. Let us preliminarily define $\tilde{C}_{\ell,s}$ the aliased version of $\Cls$. Let us define 
\begin{align}
	& \Xi^{(\ell,m)}_{N,Q}(\ell^\prime)=\kappa_{\ell,\ell^\prime}^2\sum_{r\in R_m^{Q}(\ell^\prime)}I^2_N\left(\ell,m;\ell^\prime,m+2rQ\right),\label{eq:xi1}\\
	& \Xi^{(\ell,m)}_{0;N,Q}(\ell^\prime)=\kappa_{\ell,\ell^\prime}^2\sum_{r\in R_{0,m}^{Q}(\ell^\prime)}I^2_N\left(\ell,m;\ell^\prime,m+2rQ\right),\label{eq:xi2}
\end{align}
with the proper adaptation for $m=0$
\begin{align}\notag
	& \Xi^{(\ell,0)}_{0;N,Q}(\ell^\prime)=\kappa_{\ell,\ell^\prime}^2 \sum_{r \in R_{0}^{Q}  \left(\ell+j\right)} \frac{(2\ell+1)(2(\ell+j)+1)}{2}  I^2_{N}(\ell,0,\ell+j,2rQ)\\
	&\quad +\frac{(2\ell+1)(2(\ell+2(\ell^\prime-\ell))+1)}{4}  I^2_{N}(\ell,0,\ell+2(\ell^\prime-\ell),0).\label{eq:xi3}
\end{align}
For any $\ell \geq s$, the next theorem establishes the aliases of the power spectrum $\Cls$. The proof is available in Section \ref{sec:proof}.  
\begin{theorem}\label{th:power}
	Under the assumptions of Theorem \ref{th:main} and the additional Condition~\ref{cond:summability}, let $\Xi^{(\ell,0)}_{N,Q}(\ell^\prime)$, and $\Xi^{(\ell,m)}_{0;N,Q}(\ell^\prime)$ be given by Equations \eqref{eq:xi1} and \eqref{eq:xi2}, with the proper adaptation for $m=0$ specified in \eqref{eq:xi3}. Then, each element of the aliased angular power spectrum $\left\{\tilde{C}_{\ell,s}: \ell \geq s \right\}$ is given by 
	\begin{equation*}
		\tilde{C}_{\ell;s} =\frac{2\ell+1}{4\pi} \left(  \sum_{j\in A_{\ell,s;N}}   \Xi^{(\ell,0)}_{N,Q}(\ell+j)  C_{\ell+j;s}  + \sum_{j\in B_{s;N}}\Xi^{(\ell,0)}_{0;N,Q}(\ell+j) C_{\ell+j;s}   \right).
	\end{equation*}
\end{theorem}
We consider now band-limited spherical spin random fields. As mentioned in the introduction, a spin $s$ random field is band-limited if there exists a multipole $L_0\in \mathbb{N}$, called bandwidth, such that for any $\ell > L_0$, $m=-\ell,\ldots,\ell$, $\alms=0$. In this case, the band-limited random field can be represented via the following
\begin{equation}\label{eq:bandlim}
	T_s (\vartheta,\varphi)=\sum_{\ell=0}^{L_0} \sum_{m=-\ell}^{\ell} \alms \Ylms (\vartheta,\varphi).
\end{equation}
The next result is going to show that a bandlimited spin spherical random field is alias-free. The proof is contained in Section \ref{sec:proof}.
\begin{theorem}\label{th:bandlim}
	Assume that $T_s(\vartheta,\varphi)$ is band-limited with bandwidth $L_0$ as in Equation \eqref{eq:bandlim}.	Under the assumptions of Theorem \ref{th:main}, with $N>L_0$ and $Q>L_0$, it holds that
	\begin{equation*}
		\almstilde= \alms,
	\end{equation*}
	for any $s\leq \ell \leq L_0$, and $m=-\ell,\ldots,\ell$.
\end{theorem} 
\section{A practical example}\label{sec:example}
\begin{figure}
	\caption{Sampled $\vartheta$ for both the sampling schemes (GJ: Gauss-Jacobi, EA: equiangular)}
	\begin{center}
	\includegraphics
	[scale=0.5]{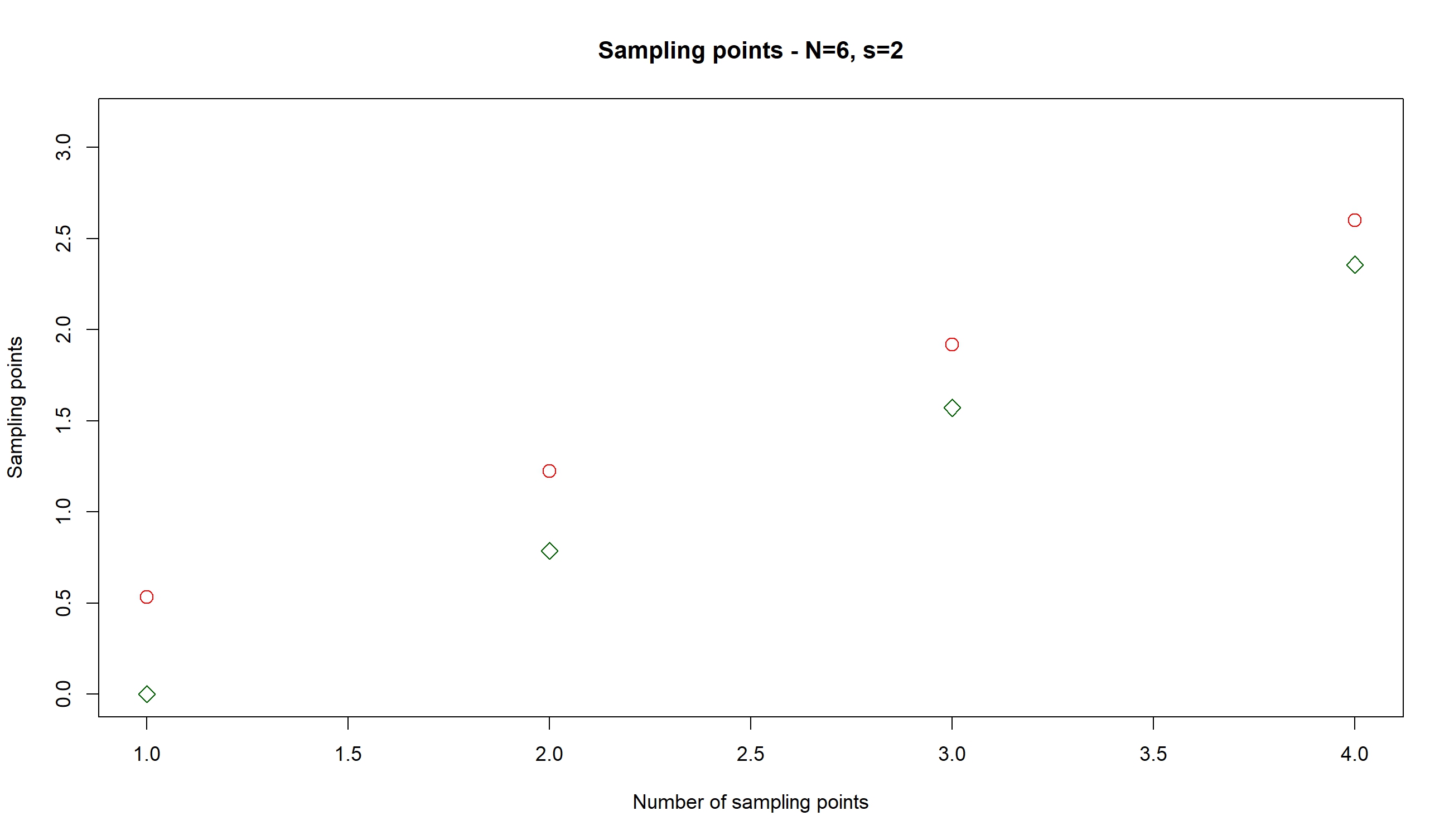}
	\label{fig:fig1}
	\end{center}
\end{figure}
In this section, we provide a simple practical example to help the reader to understand how to identify the aliases of a harmonic coefficient. This example aims to illustrate the process of alias identification for harmonic coefficients in a straightforward manner. Fixed, $s=2$, we calculate the some of the closer aliases of the harmonic coefficient $a_{2,0;2}$ and their intensities in following cases for $N=6$ and $Q=1,2$. Recall that for $N=6$ and $s=2$, we identify 4 cubature points for the colatitude $\vartheta$. For both the considered schemes, points and weights are collected in Table \ref{tab:table1}.
\begin{table}[ht]
	\centering
	\caption{Cubature nodes and weights for the Gauss--Jacobi (GJ) and equiangular (EA)
		sampling schemes, with $N=6$ and $s=2$.}
	\label{tab:table1}
	
	\begin{tabular}{c cc cc}
		\toprule
		& \multicolumn{2}{c}{\textbf{Gauss--Jacobi (GJ)}} 
		& \multicolumn{2}{c}{\textbf{Equiangular (EA)}} \\
		\cmidrule(lr){2-3} \cmidrule(lr){4-5}
		Index 
		& Nodes 
		& Weights 
		& Nodes 
		& Weights \\
		\midrule
		1 & 0.533 & 0.684 & 0.000 & 0.000 \\
		2 & 1.224 & 0.693 & 0.392 & 0.177 \\
		3 & 1.918 & 0.693 & 0.785 & 0.247 \\
		4 & 2.601 & 0.684 & 1.178 & 0.393 \\
		5 &       &       & 1.570 & 0.361 \\
		6 &       &       & 1.963 & 0.393 \\
		7 &       &       & 2.356 & 0.247 \\
		8 &       &       & 2.748 & 0.177 \\
		\bottomrule
	\end{tabular}
\end{table}
The locations over the interval $\left[0,\pi\right]$ of the sampling points $\vartheta_p$, $p=0,\ldots,3$, are portraited in Figure \ref{fig:fig1}.

Consider now the trapezoidal rule for the longitude $\varphi$. The identified aliases $a_{2+j,m+2rQ;2}$ will be thus reported in terms of the values of the indexes $j$ and $r$
\begin{itemize}
	\item \textbf{Case I:} $N=6$, $Q=1$. In this case we have $8$ and $16$ sampling points for the Gauss-Jacobi and equiangular methods respectively, as shown in Figures \ref{fig:subfig1} and \ref{fig:subfig3}. In this case, besides the primary aliases, here reported only for $j =2,3$, $r=\pm1,\pm 2$), also secondary aliases occur, located in $(j=2,r=\pm 1)$ and $(j=3,r=\pm 1)$ (Figure \ref{fig:fig2}, left panel). 
	\item \textbf{Case II:} $N=6$, $Q=2$. Here, we have $16$ and $32$ sampling points for the Gauss-Jacobi and equiangular methods respectively, as describes in Figures \ref{fig:subfig2} and \ref{fig:subfig4}. Here, only the primary locations survive, while the secondary ones are completely annihilated thanks to the choice of the parameter $Q$, as claimed by Corollary \ref{cor:1} (Figure \ref{fig:fig2}, right panel). 
\end{itemize}
\begin{figure}[htbp]
	\centering
	
	\begin{subfigure}[b]{0.45\textwidth}
		\centering
		\includegraphics[width=\textwidth]{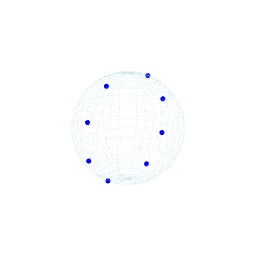}
		\caption{Gauss-Jacobi sampling: $Q=1$}
		\label{fig:subfig1}
	\end{subfigure}
	\hfill
	\begin{subfigure}[b]{0.45\textwidth}
		\centering
		\includegraphics[width=\textwidth]{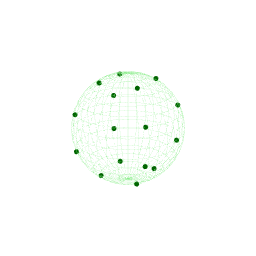}
		\caption{Gauss-Jacobi sampling: $Q=2$}
		\label{fig:subfig2}
	\end{subfigure}
	
	\vskip\baselineskip
	
	\begin{subfigure}[b]{0.45\textwidth}
		\centering
		\includegraphics[width=\textwidth]{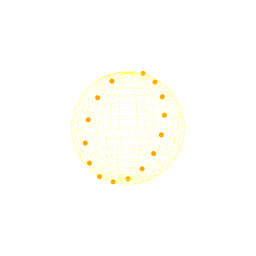}
		\caption{Equiangular sampling: $Q=1$}
		\label{fig:subfig3}
	\end{subfigure}
	\hfill
	\begin{subfigure}[b]{0.45\textwidth}
		\centering
		\includegraphics[width=\textwidth]{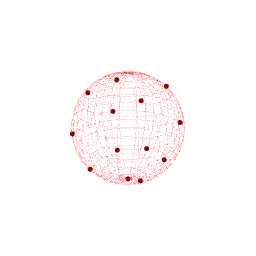}
		\caption{Equiangular sampling: $Q=2$}
		\label{fig:subfig4}
	\end{subfigure}
	
	\caption{Sampling points for Gauss-Jacobi (Figures \ref{fig:subfig1} and \ref{fig:subfig2}) and equiangular scheme (Figures \ref{fig:subfig3} and \ref{fig:subfig4}) for $Q=1$ and $Q=2$.}
	\label{fig:main}
\end{figure}
\begin{figure}
	\caption{Aliases locations ($N=6$, $s=2$, $Q=1,2$)}
	\includegraphics[width=\linewidth]{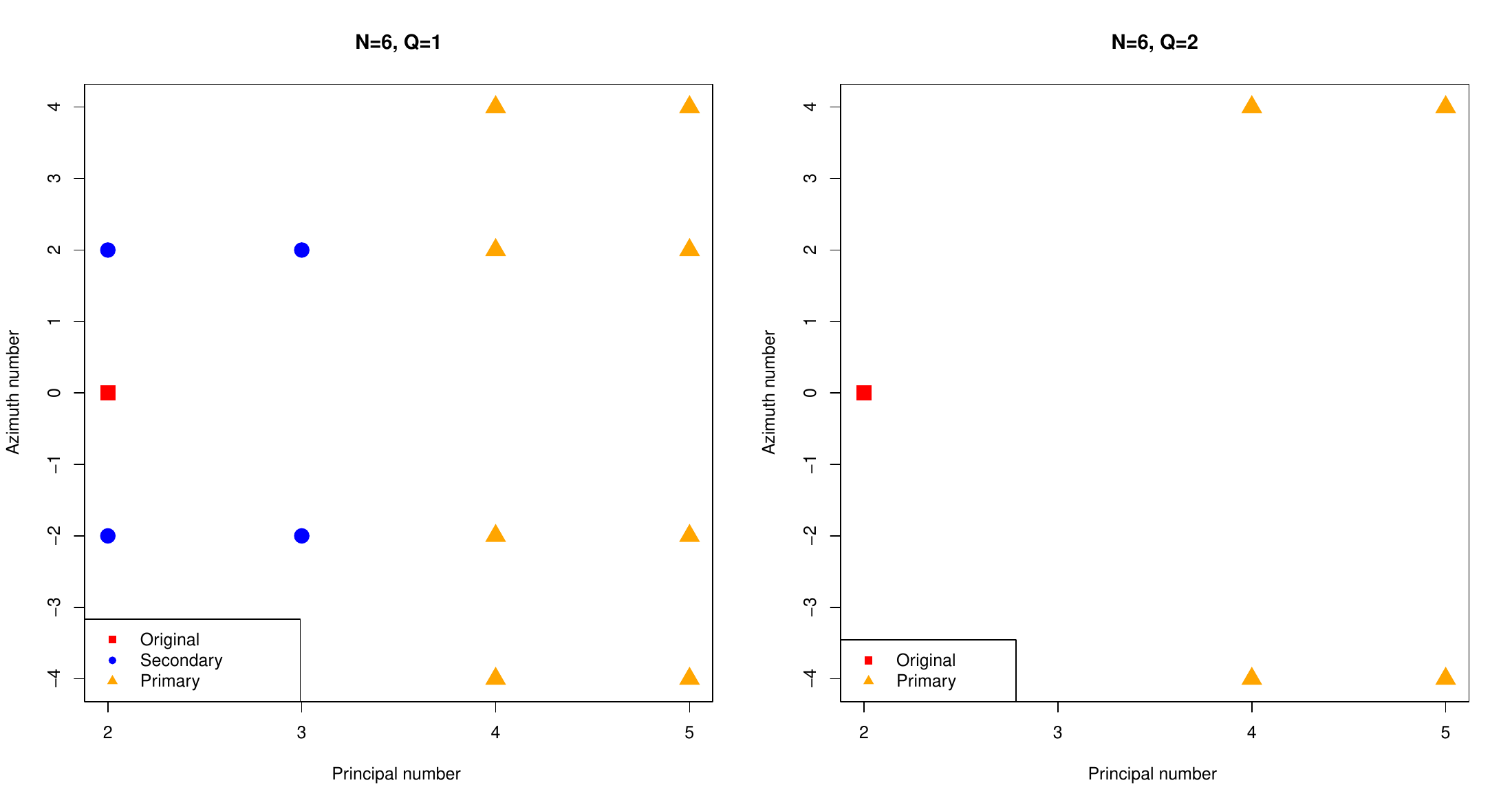}
	\label{fig:fig2}
\end{figure}
In Table \ref{tab:table2}, we collect and compare the values of the aliasing function  $\tau_2(2,0,2+j,2r)$ for both the sampling schemes. Even though the intensity is slightly higher almost everywhere for the Gauss-Jacobi scheme, it is important to remember that this effect is evaluated over approximately half the number of points used in the equiangular method. This disparity in the number of points could influence the perceived intensity differences, and the comparison should account for the differing densities of sampling points between the two methods. The Gauss-Jacobi scheme uses fewer but more strategically placed points, which can lead to more accurate integration in some cases, whereas the equiangular method distributes points evenly, providing a different sampling density that might affect the overall intensity evaluation.
\begin{table}[ht]
	\centering
	\caption{Values of the aliasing function 
		$\tau_2(2,0,2+j,2r)$ under Gauss--Legendre (GL) and equiangular (EA) sampling schemes,
		with $N=6$ and $s=2$.}
	\label{tab:table2}
	
	\begin{tabular}{lcc}
		\toprule
		Aliases $(j,r)$ 
		& \textbf{Gauss--Jacobi} 
		& \textbf{Equiangular} \\
		\midrule
		$(0,1)$   &  0.7640  &  0.6109 \\
		$(0,-1)$  &  0.7588  &  0.6112 \\
		$(1,1)$   &  0.6984  &  0.1672 \\
		$(1,-1)$  & -0.6992  & -0.1673 \\
		\midrule
		$(2,1)$   &  0.3189  &  0.2740 \\
		$(2,-1)$  &  0.3261  &  0.2728 \\
		$(2,2)$   &  0.9346  &  0.8263 \\
		$(2,-2)$  &  0.9292  &  0.8269 \\
		$(3,1)$   &  0.1981  &  0.1297 \\
		$(3,-1)$  & -0.1925  & -0.1286 \\
		$(3,2)$   &  0.3891  &  0.4492 \\
		$(3,-2)$  & -0.3968  & -0.4490 \\
		\bottomrule
	\end{tabular}
\end{table}

\section{Proofs} \label{sec:proof}
In this section, we present the proofs for the main and auxiliary results.
\begin{proof}[Proof of Theorem \ref{th:sep}]
	Under the assumptions given by Condition \ref{cond:sep}, straightforward calculations lead to  
	\begin{eqnarray*}
		\Amat & = &  \notag \sum_{p=0}^{N-1} \sum_{q=0}^{2Q-1} w_p^{\left(\vartheta\right)}  w_q^{\left(\varphi\right)} 
		Y_{u,v;s}\left(\vartheta_p,\varphi_q\right)\overline{\Ylms}\left(\vartheta_p,\varphi_q\right) \sin \vartheta_p\\
		& = & \notag \frac{\sqrt{\left(2\ell + 1\right)\left(2 u + 1 \right)}}{4\pi} \sum_{p=0}^{N-1} w_p^{\left(\vartheta\right)} 
		d^{\ell}_{m,-s}\left(\vartheta_p\right)d^{u}_{v,-s}\left(\vartheta_p\right) \sin \vartheta_p\\
		&&\times \notag \sum_{q=0}^{2Q-1}  w_q^{\left(\varphi\right)} e^{i\left(v-m\right)\varphi_q}
			\end{eqnarray*}
		Following Equations \eqref{eq:IP} and \eqref{eq:HQ}, it holds that
		\begin{eqnarray*}
		\Amat & = &  \frac{\sqrt{\left(2\ell + 1\right)\left(2 u + 1 \right)}}{4\pi}  I_{N}\left( \ell, m; u, v\right) H_Q\left(m,v\right), \notag\label{eq:asep1},
	\end{eqnarray*}
	as claimed.
\end{proof}
\begin{proof}[Proof of Theorem \ref{th:main}]
	We first focus on the trapezoidal rule for the longitude $\varphi$ (cf. the proof for \cite[Theorem 2]{dp19},). Consequently,
	\begin{align}\notag
		H_Q\left(m,v\right) &= \sum_{q=0}^{2Q-1} \frac{\pi}{Q} e^{i\left(v-m\right)q} \notag \\
		& = \begin{cases} 2\pi & \text{for }v=m+2rQ \\
			0 & \text{otherwise}
		\end{cases}, \notag 
	\end{align}
	where $r \in \mathbb{Z}$. We can now select only those values for $r$ such that $a_{u,m+2rQ;s}\neq 0$, for any $u \geq s$. This holds for $\vert m+2rQ \vert \leq u$, that is,  $r \in R_{m}^{Q}$, (cf. Equation \eqref{eq:Rm}).\\
	Thus, Equation \eqref{eq:alias1} becomes
	\begin{equation}\notag 
		\almstilde = \sum_{u \geq s} \sum_{r \in  R_{m}^{Q}\left(u\right) } \kappa_{\ell,u} I_{N}\left( \ell,m;u,m+2r Q \right) a_{u,m+2rQ;s}
	\end{equation}
	Observe that \eqref{eq:IP} becomes
	\begin{equation}\label{eq:IP2}
		I_{N}\left( \ell, m; \ell+j, m+2rQ\right) =  \sum_{p=0}^{N-1} w_p^{\left(\vartheta\right)} 
		d^{\ell}_{m,-s}\left(\vartheta_p\right)d^{\ell+j}_{m+2rQ,-s}\left(\vartheta_p\right) \sin \vartheta_p,\\
	\end{equation}
	where we have rescaled the index identifying the multipoles of the aliases by using $j=j_\ell=u-\ell$,  $j\in D_\ell =[-\ell,\infty)$.\\
	\textbf{CASE I: $j \in A_{\ell,s;N}$.}\\	 Preliminarily, we focus on the set $j \in A_{\ell,s;N}$, where the quadrature formula allows to rewrite the sum in Equation \eqref{eq:IP2} as an exact integral, so that  
	\begin{align*}
		I_{N}\left( \ell, m; \ell+j, m+2rQ\right) & = \int_{0}^{\pi} d^\ell_{m,-s} (\vartheta) d^{\ell+j}_{m+2rQ,-s} (\vartheta) \sin \vartheta \diff \vartheta\\
		& =\int_{0}^{\pi} h^{s}_{m}(\ell)h_{m+2rQ}^{s}(\ell+j) \\
		&\left( \sin \frac{\vartheta}{2}\right)^{-2(m+2rQ+s)}\left( \cos \frac{\vartheta}{2}\right)^{2(m+rQ-s)} \\ 
		&\quad \Jac{\ell+s}{-m-s}{m-s}\left(\cos\vartheta\right) \Jac{\ell+j+s}{-m-2rQ-s}{m+2rQ-s}\left(\cos\vartheta\right)\sin \vartheta \diff \vartheta\\
		& = \frac{h^{s}_{m}(\ell)h_{m+2rQ}^{s}(\ell+j)}{2^{-2s}} \\
		&\int_{-1}^{1}\left( 1-t\right)^{-(m+2rQ+s)}\left( 1+t\right)^{(m+rQ-s)} \\ 
		&\quad \Jac{\ell+s}{-m-s}{m-s}\left(t\right) \Jac{\ell+j+s}{-m-2rQ-s}{m+2rQ-s}\left(t\right)\diff t
	\end{align*}
	If we fix $r=0$, we obtain 
	\begin{align*}
			I_{N}\left( \ell, m; \ell+j, m\right) & = \frac{h^{s}_{m}(\ell)h_{m}^{s}(\ell+j)}{2^{-s}} \\
			&\int_{-1}^{1}\left( 1-t\right)^{-(m+s)}\left( 1+t\right)^{(m-s)} \\ 
			&\quad \Jac{\ell+s}{-m-s}{m-s}\left(t\right) \Jac{\ell+j+s}{-m-s}{m-s}\left(t\right)\diff t.
	\end{align*}
	In view of the orthogonality property of the Jacobi polynomials, given by Equation \eqref{eqn:orthonorm}), we have that
	\begin{align*}
		I_{N}\left( \ell, m; \ell+j, m\right) 
		& = \delta_{j}^{0}2^{2s} \left(h^{s}_{m}(\ell)\right)^2 \Lambda^{(-m-s,m-s)} _\ell\\
		&=\delta_{j}^{0} \mathcal{I}_{m,s}(\ell)
	\end{align*}
	where
	\begin{align}\label{eq:intint}
		\mathcal{I}_{m,s}(\ell)= \left(h^{s}_{m}(\ell)\right)^2 \Lambda^{(-m-s,m-s)} _\ell.
	\end{align}
It is important to note that the orthogonality property \eqref{eqn:orthonorm} is applicable only if $r=0$. If $r\neq0$, the Jacobi polynomials in the integrand function $I_{N}\left( \ell, m; \ell+j, m+2rQ\right)$ feature different parameters, case where \eqref{eqn:orthonorm} does not hold anymore.\\ 
\textbf{CASE II: $j \in B_{s;N}$.}\\	Take now $\ell \in B_{s;N}$, consider now only those coefficients with $m=0$. If additionally we fix $r=0$, Equation \eqref{eq:IP2} becomes
	\begin{align}\label{eq:cancel}
		I_{N}\left( \ell, 0; \ell+j,0\right) &=  \sum_{p=0}^{N-1} w_p^{\left(\vartheta\right)} 
		d^{\ell}_{0,-s}\left(\vartheta_p\right)d^{\ell+j}_{0,-s}\left(\vartheta_p\right) \sin \vartheta_p.
	\end{align} 
	It follows from Equation \eqref{eqn:wignparity} that
	\begin{align*}
		d^{\ell}_{0,-s}\left(\pi-\vartheta\right)d^{\ell+j}_{0,-s}\left(\pi-\vartheta\right) \sin\left(\pi-\vartheta\right)&
		= (-1)^{j}d^{\ell}_{0,-s}\left(\vartheta\right)d^{\ell+j}_{0,-s}\left(\vartheta\right) \sin\left(\vartheta\right).
	\end{align*}
	Now, observe that, since in both the sampling schemes it holds that $\vartheta_p=\vartheta_{N-1-p}$, for $p=,0,\ldots, \left[\frac{N}{2}\right]$ (cf. Section \ref{sec:parity}), the weights $w_p^{\theta}$ and $w_{N-1-p}^{\theta}$ are also equal for $p=,0,\ldots, \left[\frac{N}{2}\right]$. We remark that the in the equiangular scheme, the point identifying 0 is associated to a null weight.\\ 
	 As a consequence, if $j$ is odd, each pair of addends in the sum \eqref{eq:cancel} for $p$ and $N-1-p$ annhilate and, thus, all the aliases for $a_{\ell,0;s}$, with $r=0$ are characterized by multipoles with the same parity of $\ell$, as claimed. In this case, we have that
	 	\begin{align*}
	 	\tilde{a}_{\ell,0;s} & = \sum_{j\in A_{\ell,s;N}} \sum_{r \in R_{m}^{Q}  \left(\ell+j\right)}   \kappa_{\ell,\ell+j} I_{N}(\ell,m,\ell+j,m+2rQ) a_{\ell+j,m+2rQ;s}\\
	 	& + \sum_{j\in B_{s;N}} \sum_{r \in R_{0}^{Q}  \left(\ell+j\right)}    \kappa_{\ell,\ell+j}  I_{N}(\ell,0,\ell+j,2rQ) a_{\ell+j,2rQ;s}\\
	 	& + \sum_{j\in B_{s;N}} \kappa_{\ell,\ell+2j} I_{N}(\ell,0,\ell+2j,0) a_{\ell+2j,0;s},
	 \end{align*}
 where the last summand reflects the annihilation of all the aliases with odd $j$. 
\end{proof}
\begin{proof}[Proof of Corollary \ref{cor:1}]
Note that, if $Q \geq N-s$, for $j\in A_{\ell,s;N}$, only the aliases $a_{\ell+j,m+2rQ}$ with $r=0$ survive. Indeed, for indices $j\le N-s$, note that for any $r\neq0$
	\[
	|m+2rQ|\ge 2|r|Q-|m|\ge 2Q-\ell\ge 2Q-(\ell+j).
	\]
	Since $\ell+j\le \ell+(N-s)$ and $Q>N-s$, we have $|m+2rQ|>\ell+j$; hence the corresponding coefficient $a_{\ell+j,m+2rQ;s}$ vanishes.  
	Therefore, all terms with $r\ne0$ disappear and only $r=0$ contributes. After collapsing to $r=0$, the remaining sum \eqref{eq:IP} reduces to 
	\begin{equation*}
		I_{N}\left( \ell, m; \ell+j, m\right) =  \sum_{p=0}^{N-1} w_p^{\left(\vartheta\right)} 
		d^{\ell}_{m,-s}\left(\vartheta_p\right)d^{\ell+j}_{m,-s}\left(\vartheta_p\right) \sin \vartheta_p.
	\end{equation*}
	The product $d^{\ell}_{m,-s}\left(\vartheta_p\right)d^{\ell+j}_{m,-s}\left(\vartheta_p\right)$ is a polynomial in $\cos\vartheta_p$ of degree smaller or equal to $2\ell+j\le2(N-1)$ whenever $\ell+j\le N-1$.  
	Since the $N$-point Gauss–Jacobi rule is exact up to degree $2N-1$, the quadrature reproduces the integral exactly in this range.
	Thus, the quadrature formula allows to rewrite the last term of the previous equality as an exact integral, so that  
	\begin{align*}
		I_{N}\left( \ell, m; \ell+j, m\right) & = \int_{0}^{\pi} d^\ell_{m,-s} (\vartheta) d^{\ell+j}_{m,-s} (\vartheta) \sin \vartheta \diff \vartheta\\
		& = \delta_{j}^{0}2^{2s} \left(h^{s}_{m}(\ell)\right)^2 \Lambda^{(-m-s,m-s)} _\ell\\
		&=\delta_{j}^{0} \mathcal{I}_{m,s}(\ell),
	\end{align*}
	where $\mathcal{I}_{m,s}(\ell)$ is given by Equation \eqref{eq:intint}.
	In this case, all the secondary aliases have been annihilated. The only aliases surviving are in primariy locations, characterized by a distance from the aliased coefficient bounded by  
	\begin{align*}
	\text{dist}\left( \alms,a_{\ell+j,m+2rQ;s}\right) =& \sqrt{(j)^2+(m+2rQ-m)^2}\\
	 =&\sqrt{(N-s)^2+(2N)^2},
	\end{align*}
as claimed
\end{proof}
\begin{proof}[Proof of Theorem \ref{th:power}]
	First, note that under isotropy it holds that
	\begin{equation*}
				\Ex\left[\left\vert \alms \bar{a}_{\ell^\prime,m^\prime;s} \right \vert\right]=\Cls\delta_{\ell}^{\ell^\prime}\delta_{m}^{m^\prime}.
	\end{equation*}	
	Therefore all cross-terms with $(j,r)\ne(j',r')$ vanish, giving
	\[
	\mathbb{E}\!\left[\Big|\sum_{j,r}X_{j,r}a_{\ell+j,m+2rQ;s}\Big|^2\right]
	=\sum_{j,r}|X_{j,r}|^2C_{\ell+j;s},
	\]
	where the deterministic aliasing weights $X_{j,r}$ are those appearing in the
	definitions of $\Xi^{(\ell,m)}_{N,Q}$ and $\Xi^{(\ell,m)}_{0;N,Q}$ in 
	Equations~\eqref{eq:xi1}, \eqref{eq:xi2}, and \eqref{eq:xi3}.  
	They are explicit functions of the sampling parameters $(N,Q)$ and the harmonic indices $(\ell,m)$, and, as shown in those expressions,
	their dependence on $(j,r)$ involves only products and finite sums of Wigner $d$–functions and trigonometric factors.  
	Consequently, for fixed $(\ell,m,N,Q)$, the weights $X_{j,r}$ exhibit at most polynomial growth in $(j,r)$. Also, under the parametric assumption introduced in Equation~\eqref{eq:cls}, 
	the sequence $\{a_{\ell,m;s}\}$ is square-summable, ensuring that
	\[
	\sum_{j,r} |X_{j,r}|^2\,\mathbb{E}\!\left[|a_{\ell+j,m+2rQ;s}|^2\right]
	\le G \sum_{j,r} |X_{j,r}|^2 (1+\ell+j)^{-\alpha} < \infty.
	\]
	The series on the right-hand side remains absolutely convergent under the decay rate imposed by~\eqref{eq:cls}. Then	the right-hand side converges absolutely; hence both Fubini’s and Tonelli’s theorems apply, 
	and the expectation may be interchanged with the summations.\\
	This condition validates the following manipulations for the expectation of $|\tilde a_{\ell,m;s}|^2$, expanded as
	\[
	\begin{aligned}
		\mathbb{E}\!\left[\big|\tilde a_{\ell,m;s}\big|^2\right]
		&= \left(
		\sum_{j\in A_{\ell,s;N}} \Xi^{(\ell,m)}_{N,Q}(\ell+j)\,C_{\ell+j;s}
		+ \sum_{j\in B_{s;N}} \Xi^{(\ell,m)}_{0;N,Q}(\ell+j)\,C_{\ell+j;s}
		\right),
	\end{aligned}
	\]
	and, similarly, for $m=0$,
	\[
	\begin{aligned}
		\mathbb{E}\!\left[\big|\tilde a_{\ell,0;s}\big|^2\right]
		&= \left(
		\sum_{j\in A_{\ell,s;N}} \Xi^{(\ell,0)}_{N,Q}(\ell+j)\,C_{\ell+j;s}
		+ \sum_{j\in B_{s;N}} \Xi^{(\ell,0)}_{0;N,Q}(\ell+j)\,C_{\ell+j;s}
		\right),
	\end{aligned}
	\]
	where each term on the right-hand side is finite by virtue of~\eqref{eq:cls}.
	Hence, the interchange of expectation and infinite summation is fully justified, 
	and the resulting expression for the aliased power spectrum 
	follows rigorously from the moment bound on the coefficients.\\
	Thus, it holds that
	\begin{align*}
		\Ex\left[\left\vert \almstilde \right \vert^2\right] & =   \left(  \sum_{j\in A_{\ell,s;N}}   \frac{(2\ell+1)(2(\ell+j)+1)}{2}  C_{\ell+j;s}  \sum_{r \in R_{m}^{Q} \left(\ell+j\right)}   I^2_{N}(\ell,m,\ell+j,m+2rQ) \right.  \\
		& \left. + \sum_{j\in B_{s;N}} C_{\ell+j;s} \sum_{r \in R_{0,m}^{Q}  \left(\ell+j\right)}  \frac{(2\ell+1)(2(\ell+j)+1)}{2}   I^2_{N}(\ell,m,\ell+j,m+2rQ)  \right)\\ 
		& =   \left(  \sum_{j\in A_{\ell,s;N}}   \Xi^{(\ell,m)}_{N,Q}(\ell+j)  C_{\ell+j;s}  + \sum_{j\in B_{s;N}}\Xi^{(\ell,m)}_{0;N,Q}(\ell+j) C_{\ell+j;s}   \right),
	\end{align*}
where in the last equality we used Equations \eqref{eq:xi1} and \eqref{eq:xi2}.
For $m=0$ a similar reasoning leads to  
	\begin{align*}
		\Ex\left[\left\vert \tilde{a}_{\ell,0;s} \right \vert^2\right] & =    \left(  \sum_{j\in A_{\ell,s;N}}   C_{\ell+j;s}  \sum_{r \in R_{0}^{Q} \left(\ell+j\right)} \frac{(2\ell+1)(2(\ell+j)+1)}{2}   I^2_{N}(\ell,0,\ell+j,2rQ) \right.  \\
		&  + \sum_{j\in B_{s;N}} C_{\ell+j;s} \left(\sum_{r \in R_{0}^{Q}  \left(\ell+j\right)} \frac{(2\ell+1)(2(\ell+j)+1)}{2}  I^2_{N}(\ell,0,\ell+j,2rQ)\right. \\
		&\left.\left.+\frac{(2\ell+1)(2(\ell+2j)+1)}{2}  I^2_{N}(\ell,0,\ell+j,0)  \right)\right) \\
		& =   \left(  \sum_{j\in A_{\ell,s;N}}   \Xi^{(\ell,0)}_{N,Q}(\ell+j)  C_{\ell+j;s}  + \sum_{j\in B_{s;N}}\Xi^{(\ell,0)}_{0;N,Q}(\ell+j) C_{\ell+j;s}   \right),
	\end{align*}
	where in the last equality we used \eqref{eq:xi3}, taking care of the annilihation of more aliases happening whereas both $m$ and $r$ are equal to 0.\\
	Now we show how parity leads to exact cancellations for \(m=0\) in Wigner \(d\)-function integrals.  More in detail, we illustrate, through a concrete example, how the interplay of parity in Wigner \(d\)-functions leads to exact vanishing of certain colatitude integrals and their discrete quadrature counterparts. Take arbitrary integers \(\ell=\ell^\ast\) and \(j=j^\ast\) such that \(2\ell^\ast+j^\ast\) is odd; we show that the corresponding colatitude integral
	\[
	\int_{0}^{\pi} d^{\ell^\ast}_{0,-s}(\vartheta) d^{\ell^\ast+j^\ast}_{0,-s}(\vartheta)\sin\vartheta \diff \vartheta
	\]
	vanishes for this choice of indices, and that the quadrature sum \(I_{N^\ast}(\ell^\ast,0;\ell^\ast+j^\ast,0)\) coincides with this integral
	whenever the Gauss rule is exact on the polynomial degree in question, that is, for a proper choice of $N^\ast$.\\
	First, we establish an explicit relation to associated Legendre functions and parity.
	Up to an \(\ell\)- and \(s\)-dependent normalization factor, the Wigner \(d\)-functions with indices \((0,-s)\) reduce to associated Legendre functions in the variable \(x=\cos\vartheta\); in particular 
	\[
	P_{\ell,s}(\cos\vartheta)
	\;=\;
	(-1)^{s} \sqrt{\frac{(\ell+s)!}{(\ell-s)!}}
	d^{\ell}_{0,-s}(\vartheta),
	\]
	where \(P_{\ell,s}(x)\) denotes the associated Legendre function (cf. \cite{marpec11}). As stated by \cite{sze75}, the associated Legendre functions satisfy the parity identity 
	\[
	P_{\ell,s}(-x) = (-1)^{\ell+s}\,P_{\ell,s}(x).
	\]
	Consequently, for our two factors we obtain
	\[
	P_{\ell^\ast,s}(-x) P_{\ell^\ast+j^\ast,s}(-x) 
	= (-1)^{2\ell^\ast+j^\ast + 2s} P_{\ell^\ast}^{\,s}(x) P_{\ell^\ast+j^\ast,s}(x) 
	= -\,P_{\ell^\ast,s}(x) P_{\ell^\ast+j^\ast,s}(x),
	\]
	because \(2s\) is even for any integer \(s\) and \(2\ell^\ast+j^\ast\) is odd. Thus the product \(P_{\ell^\ast,s}(x)P_{\ell^\ast+j^\ast,s}(x)\) is an odd function of \(x\).\\
	Defining \(k_{\ell,s} = (-1)^{\ell} \sqrt{\frac{(\ell-s)!}{(\ell+s)!}}\) and passing now to the variable \(x=\cos\vartheta\), we write
	\[
	\int_{0}^{\pi} d^{\ell^\ast}_{0,-s}(\vartheta)\,d^{\ell^\ast+j^\ast}_{0,-s}(\vartheta)\sin\vartheta\,\diff\vartheta
	= k_{\ell^\ast,s}k_{\ell^\ast+j^\ast,s}\int_{0}^{\pi} P_{\ell^\ast,s}(\cos\vartheta)\,P_{\ell^\ast+j^\ast,s}(\cos\vartheta)\sin\vartheta\,\diff\vartheta.
	\]
	With \(x=\cos\vartheta\) this becomes
	\[
	k_{\ell^\ast,s}k_{\ell^\ast+j^\ast,s}\int_{-1}^{1} P_{\ell^\ast,s}(x)\,P_{\ell^\ast+j^\ast,s}(x)\,\diff x.
	\]
	Since the integrand \(P_{\ell^\ast,s}(x)P_{\ell^\ast+j^\ast,s}(x)\) is odd, the integral over the symmetric interval \([-1,1]\) vanishes.
	Note that the Gauss-Jacobi \(N^\ast\)-point quadrature used for colatitude is exact for polynomials in \(x=\cos\vartheta\)
	up to degree \(2N^\ast-1\). The product \(P_{\ell^\ast,s}(x)P_{\ell^\ast+j^\ast,s}(x)\) is thus a polynomial in \(x\) of degree at most 
	\(2\ell^\ast + j^\ast\). Hence for any \(N^\ast\) with \(2N^\ast-1 \ge 2\ell^\ast+j^\ast\)
	the discrete colatitude sum \(I_{N^\ast}(\ell^\ast,0;\ell^\ast+j^\ast,0)\)
	reproduces the continuous integral exactly and thus we obtain
	\[
	I_{N^\ast}(\ell^\ast,0;\ell^\ast+j^\ast,0) = 0 \qquad\text{for }2N^\ast-1 \ge 2\ell^\ast+j^\ast.
	\]
	Finally, to illustrate explicitly how Equations \eqref{eq:xi1}, \eqref{eq:xi2}, and \eqref{eq:xi3} operate, we examine the simplest nontrivial instance: 
	\[
	j=1, \quad r=0.
	\]
	Then the relevant combination of squared integrals reads
	\[
	\kappa_{\ell,\ell+1} \left( \left \vert I_N\left(\ell,m;\ell+1,m\right)\right \vert^2 + \left\vert I_N\left(\ell,m;\ell+1,m+2Q\right)\right \vert^2 \right).
	\]
Using	the following symmetry of the Wigner \(d\)-functions under \(\vartheta \mapsto \pi - \vartheta\): 
	\[
	d^{\ell}_{m,n}(\pi - \vartheta) = (-1)^{\ell+n} d^{\ell}_{-m,n}(\vartheta),
	\]  
we can relate the integral with \(m+2Q\) to the one with \(m\):
	\[
	\begin{split}
	I_N(\ell,m;\ell+1,m+2Q) & = (-1)^{2\ell+1} I_N(\ell,-m;\ell+1,-m-2Q) \\
	&= - I_N(\ell,-m;\ell+1,-m-2Q),
	\end{split}
	\]
	so the two terms in the sum are connected by a simple sign and index flip.
	Also, from the parity of the product of two \(d\)-functions, it holds that
\[
d^\ell_{m,s}(\vartheta)\, d^{\ell+j}_{m',s}(\vartheta)
= (-1)^{j}\, d^\ell_{m,s}(\vartheta)\, d^{\ell+j}_{m',s}(\vartheta),
\quad \text{whenever } 2\ell + j \text{ is odd}.
\]
In our case \(j=1\), so the integrals corresponding to \((m,m)\) and \((m,m+2Q)\) are related by a factor of \((-1)^1 = -1\) for the symmetric terms that would otherwise appear twice. This eliminates redundancy in the sum.
Also, we can rewrite the combination of terms differing by shifts in \(m\) or \(r\) into a single multiplicative factor,
	\[
	|I_N(\ell,m;\ell+1,m)|^2 + |I_N(\ell,-m;\ell+1,-m-2Q)|^2 = 2 |I_N(\ell,m;\ell+1,m)|^2,
	\]  
	because the two integrals have the same magnitude and are related by a symmetry transformation.  
Finally, collecting the prefactors \(\kappa_{\ell,\ell+1}\) and the multiplicity factor from the symmetric combination above, we obtain the simplified form
\[
\begin{split}
Xi^{(\ell,m)}_{N^\ast,Q}(\ell+1) &= \kappa_{\ell,\ell+1} \left( \left \vert I_{N^\ast}(\ell,m;\ell+1,m)\right \vert^2 + \left \vert I_{N^\ast}(\ell,m;\ell+1,m+2Q)\right \vert^2 \right)\\
&= 2 \,\kappa_{\ell,\ell+1} \left \vert I_{N^\ast}(\ell,m;\ell+1,m)\right \vert^2.
\end{split}
\]  
Thus, for this simple instance, we explicitly see how identities \eqref{eq:xi1}, \eqref{eq:xi2}, and \eqref{eq:xi3} merge symmetric terms and eliminate redundancies, producing the single multiplicative factor \(\Xi^{(\ell,m)}_{N^\ast,Q}(\ell+1)\) in the proper notation.
 \\
	This procedure extends straightforwardly to larger \(j\) and \(r\) by iteratively applying the same symmetries to group all equivalent terms.
\end{proof}

\begin{proof}[Proof of Theorem \ref{th:bandlim}]
	Following the lines of the proof of \cite[Theorem 6]{dp19}, this proof combines the one for Theorem \ref{th:main} and the one for Corollary \ref{cor:1}. Preliminarily, we observe that for any $\ell >L_0 $ and any $m^\prime =-\ell,\ldots,\ell$, it holds that
	\begin{equation*}
		a_{\ell^\prime,m^\prime;s}=0,
	\end{equation*} 
Also, also the harmonic coefficients for $s\leq\ell \leq L_0 $ and $m > \left \vert \ell,\right \vert$ are null. 
Indeed, recall the band-limit assumption: there exists \(L_0\in \mathbb{N}\) such that
\[
a_{\ell,m;s}=0 \qquad\text{whenever }\; \ell > L_0 \quad\text{or}\quad |m|>\ell'.
\]
Consider an aliased term \(a_{\ell+j,\;m+2rQ;\,s}\) arising in the expansion. Two independent index conditions force this coefficient to vanish.

\begin{enumerate}
	\item \emph{Band-limit in \(\ell\).} If \(\ell+j>L_0\) then \(a_{\ell+j,m+2rQ;s}=0\) by the first part of the band-limit assumption.
	
	\item \emph{Azimuth index bound.} Even when \(\ell+j\le L_0\), the coefficient is zero whenever the magnetic index violates the usual bound, that is, if
	\(\left \vert m+2rQ \right \vert |>\ell+j \), then the azimuth index is larger than the multipole and necessarily we have that \( a_{\ell+j,m+2rQ;s}=0\)
\end{enumerate}
Thus a sufficient and convenient condition for discarding all terms with \(r\neq 0\) is
	\[
	|m+2rQ|>\ell+j\quad\text{for every } r\neq 0,
	\]
	since then every aliased coefficient \(a_{\ell+j,m+2rQ;s}\) vanishes by the second rule above (or by the first rule if \(\ell+j>L_0\)). Equivalently, for fixed \((\ell,j,m)\) one may check the stronger condition
	\[
	\min\{\ell+j,L_0\} < |m+2rQ|
	\]
	which implies \(a_{\ell+j,m+2rQ;s}=0\) regardless of whether \(\ell+j\) exceeds the band-limit. In particular, if the sampling parameter \(Q\) is chosen so that \(|m+2rQ|>\ell+j\) for every \(r\neq0\) in the alias set, then all \(r\neq0\) contributions may be dropped immediately.
Thus, for any $r\neq 0$,  
	\begin{equation*}
		a_{\ell+j,m+2rQ;s}=0.
	\end{equation*} 
	The computation of the aliases of \(\tilde a_{\ell,m;s}\) is reduced to the discrete  sum
\[
\sum_{p=0}^{N-1} w_p^{(\vartheta)} \,
d^{\ell}_{m,-s}(\vartheta_p)\,d^{\ell+j}_{m,-s}(\vartheta_p)\,\sin\vartheta_p,
\]
and, by the quadrature rule, exact for polynomials in \(x=\cos\vartheta\) up to degree \(2N-1\), this sum reconstructs the continuous integral
\[ 
\int_0^\pi d^{\ell}_{m,-s}(\vartheta)\,d^{\ell+j}_{m,-s}(\vartheta)\,\sin\vartheta\,\diff \vartheta.
\]
This integral, and hence the related discrete sum, vanishes for every \(j\neq0\), so that only \(j=0\) contributes. Indeed, from the band-limit hypothesis, any aliased coefficient \(a_{\ell+j,m+2rQ;s}\) is identically zero as soon as \(\ell+j>L_0\). Thus all terms with \(\ell+j>L_0\) drop out immediately. It remains to consider indices with \(\ell+j\le L_0\).\\
We now use the standard representation of Wigner \(d\)-functions in terms of Jacobi polynomials \eqref{eqn:wigner} (cf.\ e.g.\ \cite{vmk}):, we write
\begin{equation*}
\begin{split}
	&d_{m,-s}^{\ell}\left(\vartheta\right)\\&\qquad=h_{s}^{m}(\ell)\left( \sin \frac{\vartheta}{2}\right)^{m+s}\left( \cos \frac{\vartheta}{2}\right)^{m-s} \Jac{\ell-m}{m+s}{m-s}\left(\cos\vartheta\right),
\end{split}
\end{equation*}
The same representation applies to \(d^{\ell+j}_{m,-s}\) and, therefore the integrand is, up to a positive factor \((\sin\vartheta)^{2\max(\left \vert m\right \vert,\left \vert s\right \vert)+1}\), the inner product of two Jacobi polynomials of degrees \(\ell\) and \(\ell+j\):
\[
\begin{split}
&\int_0^\pi d^{\ell}_{m,-s}(\vartheta)\,d^{\ell+j}_{m,-s}(\vartheta)\,\sin\vartheta\,d\vartheta
\\&\qquad C_{\ell,m,s}C_{\ell+j,m,s}
\int_{-1}^1 P^{(m+s,m-s)}_{n}(x)\,P^{(m+s,m-s)}_{n'}(x)\,w^{(m+s,m-s)}(x)\,\diff x.\end{split}
\]
By the orthogonality of Jacobi polynomials the last integral vanishes unless \(\ell=\ell+j\) and thus it forces \(j=0\). Thus, for any pair \((\ell,m,s)\) with \(\ell+j\le L_0\), the continuous integral, and hence its quadrature reconstruction, is zero when \(j\neq0\).
Combining the two observations (band-limit cut-off \(\ell+j>L_0\) and Jacobi orthogonality for \(\ell+j\le L_0\)) yields the desired conclusion: every aliased contribution with \(j\ne0\) vanishes, and the only surviving index in the alias-sum is \(j=0\). Finally, since quadrature is exact for the polynomial degree appearing in the integrand for a proper choice of \(N\), the discrete sum reproduces the integral exactly and thus collapses to the single \(j=0\) term.
\end{proof}
\section*{Acknowledgement and funding}
\noindent The author wishes to thank (in rigorous alphabetic order) Francesca Pistolato, Maurizia Rossi and Michele Stecconi and Nicola Turchi for useful hints and insightful suggestions.\\
This research has been partially funded by 
Progetti di Ateneo Sapienza RG1221815C353275 (2022), and RM12117A6212F538 (2021), and the PRIN project 2022 \emph{Grafia: Geometry of Random Fields and Applications}, 202284Z9E4, CUP E53D23005530006.


\begin{thebibliography}{99}
	\bibitem{planck20} Aghanim A. and the Planck Collaboration: Planck 2018 results. I. Overview
	and the cosmological legacy of Planck. \emph{A\&A}, 641:A1 (2020)
	\bibitem{at12} Atkinson, K., Han, W.: \emph{Spherical Harmonics and Approximations
		on the Unit Sphere: an Introduction}. Springer (2012)
	\bibitem{br14} Baldi P., Rossi M.: Representation of Gaussian isotropic spin random
	fields. \emph{Stochastic Process. Appl.}, 124(5), 1910--1941 (2014)
	\bibitem{dmt24} Borguin, S., Durastanti, C., Marinucci, D., Todino, A. P.: Spherical Poisson waves. \emph{Electron. J. Probab.}, 29, 1--27 (2024)
	\bibitem{cabkam05}Cabella, P., Kamionkowski, M.: Theory of Cosmic Microwave Background polarization. \emph{arXiv:astro-ph/0403392} (2005)
	\bibitem{dette05} Dette, H., Melas, V. B., Pepelyshev, A.: Optimal design for
	three-dimensional shape analysis with spherical harmonic descriptors. \emph{Ann. Statist.}, 33, 2758--2788 (2005)
	\bibitem{dette19} Dette, H., Konstantinou, M., Schorning, K., G\"osmann J.: Optimal designs for regression with spherical data.
	\emph{Electron. J. Statist.}, 13(1), 361--390 (2019)
	\bibitem{dh2003} Driscoll, J. R., Healy, D. M. J.: Computing Fourier transforms and convolutions on the sphere. \emph{Adv. Appl. Math.} 15, 202--250 (1994)
	\bibitem{dgm12} Durastanti, C., Geller, D., Marinucci, D.: Adaptive nonparametric regression on spin fiber bundles
	\emph{J. Multivariate Anal.} 104 (1), 16--38 (2012)
	\bibitem{dmp12} Durastanti, C., Marinucci, D., Peccati, G.: Normal approximations for wavelet coefficients on spherical Poisson fields. \emph{J. Math. Anal. Appl.} 409 (1), 212--227 (2014)
	\bibitem{dp19} Durastanti, C., Patschkowski, T.: Aliasing effects for random fields over spheres of arbitrary dimension. \emph{Electron. J. Stat.}, 13, 2, 3297--3335 (2019)
	\bibitem{glm2009} Geller D., Lan, X., Marinucci D.: Spin needlets spectral estimation. \emph{Electron. J. Statist.} 3, 1497--1530 (2010)
	\bibitem{gm11} Geller D., Marinucci D.: Spin wavelets on the sphere. \emph{J. Fourier Anal. Appl.}, 16(6), 840--884 (2010)
	\bibitem{gm09} Geller, D., Mayeli, A.: Nearly tight frames and space--frequency analysis on compact manifolds. \emph{Math. Z.} 263.2, 235--264 (2009)
	\bibitem{gn67} Goldberg, J. N., Newman, E. T.: Spin-$s$ spherical harmonics and $\eth$. \emph{J.
		Math. Phys.}, 8 (11), 2155--2166 (1967)
 \bibitem {gor05}
	G\'orski, K. M., Hivon, E., Banday, A. J. , Wandelt, B. D., Hansen, F. K., Reinecke, M. , Bartelmann, M.: HEALPix: A framework for high-resolution discretization and fast analysis of data distributed on the sphere.  \emph{ApJ}, 622(2), 759--771 (2005)
	\bibitem{HW10} Huffenberger, K. M., Wandelt B. D.: Fast and exact spin-s spherical harmonic transforms. \emph{ApJS} 189, 255--260 (2010)
	\bibitem{euclid20} Il\'ic S. and the Euclid Collaboration: Euclid preparation. XV. Forecasting
	cosmological constraints for the Euclid and CMB joint analysis. \emph{A\&A}, 657:A91 (2022)
	\bibitem{kks97} Kamionkowski, M.,  Kosowsky, A., Stebbins A.: Statistics of Cosmic Microwave Background polarization.
	\emph{Phys. Rev. D} 55, 7368 (1997)
	\bibitem{ls12} Leonenko, N., Sakhno, L.: On spectral representations of tensor
	random fields on the sphere. \emph{Stoch. Anal. Appl.}, 30:1, 44--66 (2012)
	\bibitem{lmrs22} Lerario, A., Marinucci, D., Rossi, M., Stecconi, M.: Geometry and topology of spin random fields. \emph{Anal. Math. Phys.}, 15, 48 (2025).
	\bibitem{ln97} Li, T.-H., North, G. R.: Aliasing effects and sampling theorems of spherical random fields when sampled on a finite grid. \emph{Ann. Inst.
		Statist. Math.}, 49(2), 341--354 (1997)
	\bibitem{lns97} Li, T.-H., North, G. R., Shen, S. S.: Aliased power of a stochastic temperature on a sphere. \emph{J. Geophys. Res.}, 102(D4), 4475--4486 (1997)
	\bibitem{mal11} Malyarenko, A.: Invariant random fields in vector bundles and application to
	cosmology. \emph{Ann. Inst. Henri Poincar\'e Probab. Stat.}, 47(4), 1068-–1095 (2011)
	\bibitem{mal13} Malyarenko A.: \emph{Invariant Random Fields on Spaces with a Group Action.} Springer (2013)
	\bibitem{marpec11} Marinucci, D., Peccati, G.: \emph{Random Fields on the Sphere}. Cambridge (2011)
	\bibitem{MW11} McEwen, J. D., Wiaux, Y.: A novel sampling theorem on the sphere. \emph{IEEE Trans. Signal Process.} 59, 12 (2011)
	\bibitem{euclid24} Mellier T. and the Euclid Collaboration: Euclid. I. Overview of the Euclid mission. \emph{A\&A}, 697, A1 (2025)
	\bibitem{m07} M\"uller, W. G.: \emph{Collecting Spatial Data: Optimal Design of Experiments for Random Fields}. Springer Verlag (2007)
	\bibitem{np66} Newman E. T., Penrose R.: Note on the Bondi-Metzner-Sachs group. \emph{J.
		Math. Phys.}, 7, 863-–870, (1966)
	\bibitem{ps24} Pistolato, F., Stecconi, M.: Expected Lipschitz-Killing curvatures for spin random fields and other non-isotropic fields. \emph{arXiv:2406.04850} (2024)
	\bibitem{pm96} Proakis J. G., Manolakis D. G.: \emph{Digital Signal Processing: Principles, Algorithms, and Applications.} Prentice-Hall International (1996)
	\bibitem{risbo} Risbo, T.: Fourier transform summation of Legendre series and $D$-functions. \emph{J.Geod.}, 70, 7, 383--396 (1996)
	\bibitem{sku86} Skukowsky, W.; A quadrature formula over the sphere with applications to high resolution spherical harmonic analysis. \emph{J. Geod.}, 60, 1--14 (1986)
	\bibitem{stec22} Stecconi, M.: Isotropic random spin weighted functions on {$\mathbb{S}^2$} vs isotropic random fields on {$\Bbb{S}^3$}. \emph{Theory Prob. Math. Statist.} 107, 77--109 (2022)
	\bibitem{stobur93} Stoer J., Bulirsch R.: \emph{Introduction to Numerical Analysis}. Springer-Verlag, (1993)
	\bibitem{sze75} Szego, G.: \emph{Orthogonal Polynomials}. Amer. Math. Soc. Colloq. Publ., 4th edition (1975)
	\bibitem{todino23} Todino, A. P.: No smooth phase transition for the nodal length of band-limited spherical random fields. \emph{Stochastic Process. Appl.} 169, 104273 (2024)
	\bibitem{vmk} Varshalovich, D. A., Moskalev, A. N. and Khersonskii, V. K.: \emph{Quantum
		Theory of Angular Momentum}. World Scientific (1988)
	\bibitem{wg01} Wandelt, B. D., and Gorski, K. M.: Fast convolution on the sphere. \emph{Phys. Rev. D},  63, 123002 (2001)
	\bibitem{wjv07} Wiaux, Y., Jacques, L. , and P. Vandergheynst . Fast spin $\pm s$ spherical harmonic transforms and applicationin cosmology. \emph{J. Comput. Phys.} 226, 2359--2371 (2007)
\end{thebibliography}
\end{document}